\documentclass[a4paper,11pt]{article}

\usepackage[T1]{fontenc}
\usepackage{lmodern}

\usepackage{dsfont}

\usepackage[latin1]{inputenc}
\usepackage{amsmath}
\usepackage{amsthm}
\usepackage{amssymb}
\usepackage{mathrsfs}
\usepackage{graphicx}
\usepackage[all]{xy}
\usepackage{hyperref}

\usepackage{makeidx}

\usepackage{stmaryrd}
\usepackage{caption}

\usepackage{abstract} 

\newtheorem{thm}{Theorem}[section]
\newtheorem{cor}[thm]{Corollary}
\newtheorem{claim}[thm]{Claim}
\newtheorem{fact}[thm]{Fact}

\newtheorem{lemma}[thm]{Lemma}
\newtheorem{prop}[thm]{Proposition}

\theoremstyle{definition}
\newtheorem{definition}[thm]{Definition}
\newtheorem{ex}[thm]{Example}
\newtheorem{remark}[thm]{Remark}
\newtheorem{question}[thm]{Question}

\def\rquotient#1#2{%
	\makeatletter
	\raise.3ex\hbox{$#1$}/\lower.3ex\hbox{$#2$}%
	\makeatother
}	

\makeatletter
\newcommand{\subjclass}[2][2010]{%
	\let\@oldtitle\@title%
	\gdef\@title{\@oldtitle\footnotetext{#1 \emph{Mathematics subject classification.} #2}}%
}
\newcommand{\keywords}[1]{%
	\let\@@oldtitle\@title%
	\gdef\@title{\@@oldtitle\footnotetext{\emph{Key words and phrases.} #1.}}%
}
\makeatother

\newcommand{\Address}{{
		\bigskip
		\small
		
\noindent\textsc{Universit\'e de Montpellier\\ 
Institut Math\'ematiques Alexander Grothendieck\\
Place Eug\`ene Bataillon\\
34090 Montpellier (France)}\par\nopagebreak
\noindent \textit{E-mail address}: \texttt{anthony.genevois@umontpellier.fr}
		
}}

\makeindex

\title{Polynomial hyperbolicity and products of free groups}
\date{\today}
\author{Anthony Genevois}


\begin{document}

\maketitle

\begin{abstract}
In this article, we define a locally finite graph $X$ as \emph{$\eta$-polynomially hyperbolic} if there exists a Lipschitz map $\varphi : X \to Z$ to some hyperbolic space $Z$ satisfying the following condition: there exists $C \geq 0$ such that
$$|B(p,R_1) \cap \varphi^{-1} (B(q,R_2))| \leq (C R_1)^{\eta(C R_2)} \text{ for all } p,q \in X, R_1,R_2 \geq 0.$$
The picture to keep in mind is that coarse fibres of $\varphi$ have polynomial growth with a degree coarsely controlled by $\eta$ as the thickness of the fibres grows. The map $\eta$ quantifies how brutal we have to be in order to turn $X$ into a hyperbolic space. Our main result is that, among cocompact special groups, being $\mathrm{lin}$-polynomially hyperbolic amounts not to contain $\mathbb{F}_2 \times \mathbb{F}_2$ as a subgroup. Consequently, containing $\mathbb{F}_2 \times \mathbb{F}_2$ as a subgroup turns out to be quasi-isometric invariant for cocompact special groups. 
\end{abstract}

\tableofcontents

\section{Introduction}

\noindent
In geometric group theory, given a geodesic metric space $X$, it is often desirable to construct a hyperbolic model that extracts the negative curvature of $X$. Typically, the strategy is to identify the obstruction of $X$ from being hyperbolic and to cone-off $X$ in order to collapse this obstruction, producing a(n equivariant) coarsely surjective Lipschitz map from $X$ to a hyperbolic space. Of course, the goal is to cone-off subspaces of $X$ that are as small as possible, in order to keep as much information about the geometry of $X$ as possible. Loosely speaking, the less we need to collapse, the more hyperbolic our group is. In this article, we propose to quantify this idea through the following definition:

\begin{definition}
Let $X$ be a locally finite graph and $Y$ a metric space. Given a function $F : [0,+ \infty)^2 \to [0,+ \infty)$, a map $\varphi : X \to Y$ is \emph{$F$-gentle} if there exists a constant $C\geq 0$ such that
$$| B(p,R_1) \cap \varphi^{-1}(B(q,R_2)) | \leq C F(CR_1,CR_2)$$
for all $R_1,R_2 \geq 0$, $p \in X$, and $q \in Y$. 
\end{definition}

\noindent
The idea to keep in mind is that the first coordinate of the map $F(\cdot,\cdot)$ controls the growth of coarse fibres and the second coordinate of $F(\cdot,\cdot)$ controls how this growth evolves as the thickness of the fibres increases. Roughly speaking, $F$ quantifies how brutal we have to be in order to turn our metric space into a hyperbolic space. See Section~\ref{section:gentle} for examples. In this article, we mainly focus on functions $F(x,y)=x^{\eta(y)}$. Then, coarse fibres have polynomial growth and $\eta$ quantifies how the degree coarsely evolves as the thickness of the fibres increases.

\begin{definition}
A locally finite graph $X$ is \emph{$\eta$-polynomially hyperbolic} if there exists a Lipschitz map $X \to Y$ to a hyperbolic space that is $x^{\eta(y)}$-gentle. 
\end{definition}

\noindent
For instance, every graph $X$ of polynomial growth is $\mathrm{bnd}$-polynomially hyperbolic since $X \twoheadrightarrow \mathrm{pt}$ is $x^n$-gentle for some $n \geq 1$. Similarly, given a graph $X$ of polynomial growth and a hyperbolic graph $Y$ of bounded degree, $X \times Y$ is $\mathrm{lin}$-polynomially hyperbolic since the projection map $X \times Y \twoheadrightarrow Y$ is $x^y$-gentle. It is worth noticing that, because every graph embeds in a hyperbolic horoball (see Example~\ref{ex:ExpPolynomial}), it turns out that every locally finite graph is $\mathrm{exp}$-polynomially hyperbolic. 

\medskip \noindent
Thus, every locally finite graph is $\eta$-polynomially hyperbolic for some $\eta$ between bounded and exponential. The main result of our article is that the lamplighter graph $\mathrm{Lam}(\mathbb{Z})$, or equivalently the lamplighter group $\mathbb{Z}_2 \wr \mathbb{Z}$, is not $\eta$-polynomially hyperbolic for some polynomial $\eta$, which motivates the idea that $\mathrm{Lamp}(\mathbb{Z})$ is very far from being hyperbolic. Since $\mathrm{Lamp}(\mathbb{Z})$ quasi-isometrically embeds into the product of free groups $\mathbb{F}_2 \times \mathbb{F}_2$, the same conclusion holds for $\mathbb{F}_2 \times \mathbb{F}_2$. 

\begin{thm}
The lamplighter $\mathrm{Lamp}(\mathbb{Z})$, and consequently the product of free groups $\mathbb{F}_2 \times \mathbb{F}_2$, is not $\mathrm{pol}$-polynomially hyperbolic.
\end{thm}

\noindent
It turns out that, in several cases of interest, there is a gap in the behaviour of polynomial hyperbolicity for finitely generated groups: either our group is $\mathrm{lin}$-polynomially hyperbolic or it cannot be $\mathrm{pol}$-polynomially hyperbolic since it contains a quasi-isometrically embedded copy of $\mathbb{F}_2 \times \mathbb{F}_2$. As a consequence, polynomial hyperbolicity allows us to characterise geometrically when some finitely generated groups contain $\mathbb{F}_2 \times \mathbb{F}_2$ as a subgroup. 

\medskip \noindent
We illustrate this phenomenon with two families of groups. The first groups we consider are \emph{cocompact special groups}, as introduced in \cite{MR2377497}. Our main result in this direction~is:

\begin{thm}\label{thm:IntroSpecial}
Let $G$ be a cocompact special group. The following are equivalent:
\begin{itemize}
	\item $G$ contains $\mathbb{F}_2 \times \mathbb{F}_2$ as a subgroup;
	\item $G$ contains $\mathbb{F}_2 \times \mathbb{F}_2$ as an undistorted subgroup;
	\item $\mathbb{F}_2 \times \mathbb{F}_2$ quasi-isometrically embeds into $G$;
	\item $\mathrm{Lamp}(\mathbb{Z})$ quasi-isometrically embeds into $G$;
	\item $G$ is not $\mathrm{pol}$-polynomially hyperbolic;
	\item $G$ is not $\mathrm{lin}$-polynomially hyperbolic.
\end{itemize}
\end{thm}

\noindent
As an immediate consequence, it follows that:

\begin{cor}\label{cor:IntroQIspecial}
Given two quasi-isometric cocompact special groups $G$ and $H$, $\mathbb{F}_2 \times \mathbb{F}_2$ is a subgroup of $G$ if and only if $\mathbb{F}_2 \times \mathbb{F}_2 $ is a subgroup of $H$. 
\end{cor}

\noindent
Corollary~\ref{cor:IntroQIspecial} completes the facts that, among cocompact special groups, containing $\mathbb{Z}^2$ or $\mathbb{F}_2 \times \mathbb{Z}$ is also invariant under quasi-isometries (characterising hyperbolicity and toral relative  hyperbolicity respectively; see\cite{MR2541383,SpecialRH}). (See Question~\ref{question:IntroPower}.) However, it is worth mentioning that our corollary is no longer true among cubulable groups, since there exist uniform lattices in products of two trees that are commutative transitive \cite{MR2174099}, proving that $\mathbb{F}_2 \times \mathbb{F}_2$ is quasi-isometric to a cubulable group that does not contain $\mathbb{F}_2 \times \mathbb{F}_2$ (or even $\mathbb{F}_2 \times \mathbb{Z}$) as a subgroup.  

\medskip \noindent
Theorem~\ref{thm:IntroSpecial} can be applied to explicit examples of groups. For instance, it provides the following characterisation for \emph{right-angled Artin groups}. Recall that, given a graph $\Gamma$, the right-angled Artin group $A(\Gamma)$ is defined by the presentation
$$\langle V(\Gamma) \mid [u,v]=1 \text{ for every edge } \{u,v\} \in E(\Gamma) \rangle.$$
Part of our corollary can also be found in \cite[Theorem~E]{Oussama} for two-dimensional right-angled Artin groups, with a proof based on a completely different method.

\begin{cor}\label{cor:RAAG}
Let $\Gamma$ be a finite graph. The following statements are equivalent:
\begin{itemize}
	\item $\Gamma$ contains an induced $4$-cycle;
	\item $A(\Gamma)$ contains $\mathbb{F}_2 \times \mathbb{F}_2$ as a subgroup;
	\item $\mathbb{F}_2 \times \mathbb{F}_2$ quasi-isometrically embeds into $A(\Gamma)$;
	\item $\mathrm{Lamp}(\mathbb{Z})$ quasi-isometrically embeds into $A(\Gamma)$.
\end{itemize}
\end{cor}

\begin{proof}
The equivalence between the first two items is given by \cite[Theorem~3.7]{MR2475886}. The rest of the corollary is given by Theorem~\ref{thm:IntroSpecial}. 
\end{proof}

 \noindent
As another application of Theorem~\ref{thm:IntroSpecial}, let us consider \emph{graph braid groups}. Recall that, given a connected topological graph $\Gamma$ and an integer $n \geq 1$, the graph braid group $B_n(\Gamma)$ is the fundamental group of the configuration space
$$\{ (x_1, \ldots, x_n) \in \Gamma^n \mid \forall i \neq j, x_i \neq x_j \} / \mathrm{Sym}(n)$$
where the symmetric group $\mathrm{Sym}(n)$ permutes the coordinates. 

\begin{cor}
Let $\Gamma$ be connected topological graph. The following statements are equivalent:
\begin{itemize}
	\item for all disjoint connected subgraphs $\Lambda_1,\Lambda_2 \subset \Gamma$ and for all integers $n_1,n_2 \geq 1$ satisfying $n_1+n_2 \leq n$, there exists $i \in \{1,2\}$ such that one of the following conditions holds:
	\begin{itemize}
		\item $n_i=1$ and $\chi(\Lambda_i) \geq 0$;
		\item $n_i=2$ and $\Lambda_i$ is either a segment, or a circle, or a star with three arms;
		\item $n_i \geq 3$ and $\Lambda_i$ is either a segment or a circle;
	\end{itemize}
	\item $B_n(\Gamma)$ does not contain $\mathbb{F}_2 \times \mathbb{F}_2$ as a subgroup;
	\item $\mathbb{F}_2 \times \mathbb{F}_2$ does not quasi-isometrically embeds into $B_n(\Gamma)$;
	\item $\mathrm{Lamp}(\mathbb{Z})$ does not quasi-isometrically embeds into $B_n(\Gamma)$.
\end{itemize}
\end{cor}

\begin{proof}
We know from \cite[Theorem~1.1 and Proposition~6.7]{MR4922688} exactly when a graph braid group does not contain $\mathbb{F}_2 \times \mathbb{F}_2$ as a subgroup, which corresponds to the graphical criterion given above according to \cite[Lemmas~4.3 and~4.14]{MR4227231}. The rest of the theorem is given by Theorem~\ref{thm:IntroSpecial}. 
\end{proof}

\noindent
After cocompact special groups, the second family of groups we use in order to illustrate this gap in polynomial hyperbolicity is given by \emph{graph products of groups}. Recall that, given a graph $\Gamma$ and a collection of groups $\mathcal{G}= \{G_u \mid u \in V(\Gamma)\}$ indexed by $V(\Gamma)$, the graph product $\Gamma \mathcal{G}$ is given by the relative presentation
$$\langle G_u \ (u \in V(\Gamma)) \mid [G_u,G_v]=1 \text{ for all edges } \{u,v\} \in E(\Gamma) \rangle,$$
where $[G_u,G_v]=1$ is a shorthand for: $[g,h]=1$ for all $g \in G_u$ and $h \in G_v$. Notice that graph products of infinite cyclic groups (resp.\ cyclic groups of order two) coincide with right-angled Artin (resp.\ Coxeter) groups. Our main result in the study of graph products is:

\begin{thm}\label{thm:IntroGP}
Let $\Gamma$ be a finite graph and $\mathcal{G}$ a collection of non-trivial groups of polynomial growth. The following statements are equivalent:
\begin{itemize}
	\item $\Gamma$ contains one of the graphs from Figure~\ref{GraphsIntro} as an induced subgraph;
	\item $\Gamma \mathcal{G}$ contains $\mathbb{F}_2 \times \mathbb{F}_2$ as a subgroup;
	\item $\mathbb{F}_2 \times \mathbb{F}_2$ quasi-isometrically embeds into $\Gamma \mathcal{G}$;
	\item $\mathrm{Lamp}(\mathbb{Z})$ quasi-isometrically embeds into $\Gamma \mathcal{G}$;
	\item $\Gamma \mathcal{G}$ is not $\mathrm{pol}$-polynomially hyperbolic;
	\item $\Gamma \mathcal{G}$ is not $\mathrm{lin}$-polynomially hyperbolic.
\end{itemize}
\end{thm}
\begin{figure}[h!]
\begin{center}
\includegraphics[width=0.7\linewidth]{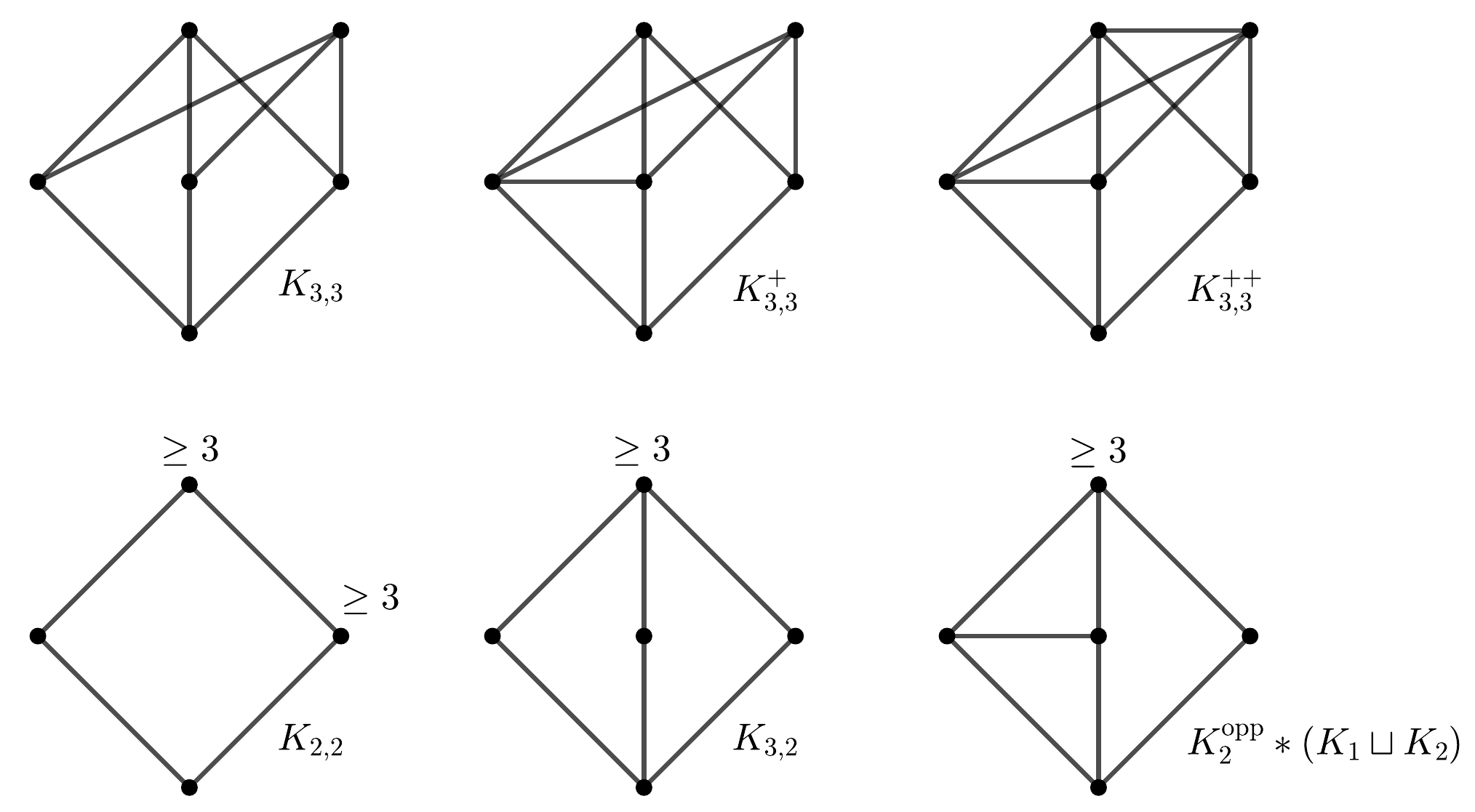}
\caption{}
\label{GraphsIntro}
\end{center}
\end{figure}

\noindent
Again, Theorem~\ref{thm:IntroGP} can be applied to explicit examples of groups. Of course, we recover Corollary~\ref{cor:RAAG} for right-angled Artin groups. For right-angled Coxeter groups, we deduce an analogous statement:

\begin{cor}\label{cor:RACGs}
Let $\Gamma$ be a finite graph. The following statements are equivalent:
\begin{itemize}
	\item $\Gamma$ contains an induced copy of $K_{3,3}$, $K_{3,3}^+$, or $K_{3,3}^{++}$;
	\item $C(\Gamma)$ contains $\mathbb{F}_2 \times \mathbb{F}_2$ as a subgroup;
	\item $\mathbb{F}_2 \times \mathbb{F}_2$ quasi-isometrically embeds into $C(\Gamma)$;
	\item $\mathrm{Lamp}(\mathbb{Z})$ quasi-isometrically embeds into $C(\Gamma)$.
\end{itemize}
\end{cor}

\noindent
Here, $K_{3,3}$ refers to the complete bipartite graph, which is also the join of two triples of isolated vertices; $K_{3,3}^+$ refers to the join of a triple of isolated vertices with $K_1 \sqcup K_2$ (i.e.\ the disjoint union of an isolated vertex with a single edge); and finally, $K_{3,3}^{++}$ refers to the join of two copies of $K_1 \sqcup K_2$. See Figure~\ref{GraphsIntro}. 

\begin{proof}[Proof of Corollary~\ref{cor:RACGs}.]
The equivalence between the first two items is given by \cite[Theorem~1.1 and Proposition~6.1]{MR4922688}. The rest of the corollary is given by Theorem~\ref{thm:IntroGP}. 
\end{proof}

\noindent
We conclude this introduction with the following question, which we leave open:

\begin{question}\label{question:IntroPower}
Let $n,m \geq 0$ be two integers. Among cocompact special group, is containing $\mathbb{F}_2^n \times \mathbb{Z}^m$ as a subgroup a quasi-isometric invariant?
\end{question}

\noindent
As already said, the answer is positive for $(n,m)$ equal to $(0,2)$ according to \cite{MR2541383} (characterising hyperbolicity), to $(1,2)$ according to \cite{SpecialRH} (characterising toral relative hyperbolicity), and finally to $(2,0)$ according to Corollary~\ref{cor:IntroQIspecial} (characterising polynomial hyperbolicity). A positive answer is also obtained in \cite{Oussama} for $m=0$ and $n$ arbitrary but in the more restricted family of two-dimensional right-angled Artin groups. 

\medskip \noindent
Polynomial hyperbolicity does not allow us to distinguish $\mathbb{F}_2 \times \mathbb{F}_2$ and $\mathbb{F}_2 \times \mathbb{F}_2 \times \mathbb{F}_2$. However, the more general idea that $\mathbb{F}_2^n$ can be mapped to a hyperbolic space only with large coarse fibres could apply. We expect polynomial hyperbolicity to be a first step in a hierarchy of properties that would allow us to distinguish cocompact special groups containing $\mathbb{F}_2 \times \mathbb{Z}^m$ for various values of $n$ and $m$.

\paragraph{Acknowledgements.} I am grateful to Oussama Bensaid who shared with me some of his nice results from \cite{Oussama}. Generalising \cite[Theorem~E]{Oussama} from two-dimensional right-angled Artin groups to a more general setting motivated this work.

\section{Gentle maps}\label{section:gentle}

\noindent
Intuitively, what we want to do is to quantify how brutal we have to be in order to turn a metric space into a hyperbolic space. For this purpose, we introduce the following notion:

\begin{definition}\label{definition:Gentle}
Let $X$ be a locally finite graph and $Y$ a metric space. Given a function $F : [0,+ \infty)^2 \to [0,+ \infty)$, a map $\varphi : X \to Y$ is \emph{$F$-gentle} if there exists a constant $C\geq 0$ such that
$$| B(p,R_1) \cap \varphi^{-1}(B(q,R_2)) | \leq C F(CR_1,CR_2)$$
for all $R_1,R_2 \geq 0$, $p \in X$, and $q \in Y$. 
\end{definition}

\noindent
The idea to keep in mind is that the first coordinate of the map $F(\cdot,\cdot)$ controls the growth of coarse fibres and the second coordinate of $F(\cdot,\cdot)$ controls how this growth evolves as the thickness of the fibres grows. 

\medskip \noindent
Typically, given a locally finite graph, the goal will be to find a function $F$ as small as possible such that there exists a Lipschitz $F$-gentle map from our graph to some hyperbolic space.

\begin{remark}
Definition~\ref{definition:Gentle} naturally extends to arbitrary metric measure spaces. In view of the applications we are interested in, namely finitely generated groups, we restrict ourselves to locally finite graphs for simplicity. 
\end{remark}

\paragraph{First examples.} In order to illustrate the notion of gentle maps, we record below a few natural examples. 

\begin{ex}\label{ex:ToApoint}
Given two graphs of bounded degree $X$ and $Y$, clearly every map $X \to Y$ is $F$-gentle for $F(x,y)= \mathrm{growth}_X(x)$ where $\mathrm{growth}_X : R \mapsto \max_{p \in V(X)} |B(p,R)|$. Thus, every graph $X$ of bounded degree can be mapped $F$-gently to the hyperbolic space $\{\mathrm{pt}\}$ with $F(x,y)= \mathrm{growth}_X(x)$, and a fortiori with $F(x,y)= \exp(x)$. 
\end{ex}

\begin{ex}\label{ex:ExpPolynomial}
Recall that, given a graph $X$, the \emph{horoball} $\mathcal{H}(X)$ is the graph whose vertex-set is $V(X) \times \mathbb{N}$ and whose edges connect $(x,n)$ to $(y,m)$ if either $y=x$ and $m=n+1$ or $m=n$ and $d_X(x,y) \leq 2^{-n}$. Horoballs are always hyperbolic \cite[Theorem~3.8]{MR2448064}. As an easy consequence of \cite[Lemma~3.10]{MR2448064}, if $X$ is locally finite, then the canonical embedding $\varphi : X \to \mathcal{H}(X)$ is $F$-gentle for $F(x,y)= \mathrm{growth}_X(\min(x,2^y))$.
\end{ex}

\begin{ex}\label{ex:ProjectionProduct}
Given two locally finite graphs $X$ and $Y$, the projection map $X \times Y \to Y$ is $F$-gentle with $F(x,y)=\mathrm{growth}_X(x)\mathrm{growth}_Y(\min(x,y))$.  For instance, $\mathbb{F}_2 \times \mathbb{Z}^n \to \mathbb{F}_2$ is $2^y\min(x,y)^n$-gentle. 
\end{ex}

\begin{ex}
Let $\mathbb{F}_2= \langle a,b \mid \rangle$ be the free group of rank two. The canonical map $\mathrm{Cayl}(\mathbb{F}_2, \{a,b\}) \to \mathrm{Cayl}(\mathbb{F}_2, \langle a \rangle \cup \langle b \rangle)$ is $x^y$-gentle. This can be checked by direct computation, or as a consequence of Lemma~\ref{prop:QuasiSyl} below. 
\end{ex}

\paragraph{Polynomial hyperbolicity.} We will mainly focus on $F$-gentle maps with $F(x,y)=x^{\eta(y)}$ for some functions $\eta : [0,+ \infty) \to [0,+ \infty)$. In this case, the intuition to keep in mind is that coarse fibres have polynomial growth and that the map $\eta$ controls how the degree evolves when the thickness of the fibre increases.  

\begin{definition}
A locally finite graph $X$ is \emph{$\eta$-polynomially hyperbolic} if there exists a Lipschitz map $X \to Y$ to a hyperbolic space that is $x^{\eta(y)}$-gentle. 
\end{definition}

\noindent
As justified by our next lemma, polynomial hyperbolicity is clearly preserved by quasi-isometries. Consequently, it makes sense to refer to a finitely generated group as being ($\eta$-)polynomially hyperbolic whenever its Cayley graphs (constructed from finite generating sets) are ($\eta$-)polynomially hyperbolic. 

\begin{lemma}\label{lem:GentleStableQI}
Let $X,Y$ be two graphs of bounded degree and $\rho : X \to Y$ a quasi-isometric embedding. If $\varphi : Y \to Z$ is an $F$-gentle map to some metric space $Z$, then $\varphi \circ \rho$ is also $F$-gentle. Consequently, if $Y$ is $\eta$-polynomially hyperbolic for some $\eta$, then $X$ must be $\eta$-polynomially hyperbolic as well. 
\end{lemma}

\begin{proof}
Fix $A>0$, $B \geq 0$, and $N \geq 0$ such that $\rho$ is an $(A,B)$-quasi-isometric embedding with at most $N$ points in each pre-image. Given two points $p \in V(X)$, $q \in Z$ and two radii $R_1,R_2 \geq 0$, we deduce from the inclusion $B(p,R_1) \subset \rho^{-1} (B(\rho(p),AR_1+B))$ that
$$B(p,R_1) \cap (\varphi \circ \rho)^{-1} ( B(q,R_2)) \subset \rho^{-1} \left( B(\rho(p),AR_1+B) \cap \varphi^{-1}(B(q,R_2)) \right).$$
Since $\varphi$ is $F$-gentle, we find a uniform constant $C \geq 0$ such that 
$$|B(p,R_1) \cap (\varphi \circ \rho)^{-1} ( B(q,R_2))| \leq NC F(C(AR_1+B), R_2).$$
The desired conclusion follows. 
\end{proof}

\noindent
Let us mention a few examples of polynomially hyperbolic graphs and groups.

\begin{ex}
A locally finite graph is bounded-polynomially hyperbolic if and only if it has polynomial growth. (For the converse, this follows from Example~\ref{ex:ToApoint}.)
\end{ex}

\begin{ex}
As a consequence of Example~\ref{ex:ExpPolynomial}, every graph of bounded degree is exp-polynomially hyperbolic. 
\end{ex}

\begin{ex}
As a consequence of Example~\ref{ex:ProjectionProduct}, for every hyperbolic group $G$ and for every integer $n \geq 0$, $G \times \mathbb{Z}^n$ is linear-polynomially hyperbolic. 
\end{ex}

\paragraph{Quasi-syllabic collections.} In order to prove that a given graph $X$ is polynomially hyperbolic, typically the strategy will be to map $X$ to some hyperbolic graph obtained by conning-off some subgraphs of polynomial growth. In this article, we use the following definition of cone-offs:

\begin{definition}
Let $X$ be a graph and $\mathcal{P}$ a collection of vertex-sets. The \emph{cone-off} $\mathrm{ConeOff}(X, \mathcal{P})$ is the graph obtained from $X$ by adding an edge between any two vertices that belong to a common subset in $\mathcal{P}$. 
\end{definition}

\noindent
There already exist several efficient criteria available in the literature in order to prove that specific cone-offs are hyperbolic. In addition, we need to show that the canonical projections to our cone-offs are sufficiently gentle. Proposition~\ref{prop:QuasiSyl} below provides the criterion we will need in our applications. It is based on the following definition:

\begin{definition}
Let $X$ be a graph. A collection of vertex-sets $\mathcal{P}$ is \emph{quasi-syllabic} if there exist $A>0$ and $B \geq 0$ such that any two vertices of $X$ can be connected by some $(A,B)$-quasigeodesic in $\mathrm{ConeOff}(X,\mathcal{P})$ that extends to an $(A,B)$-quasigeodesic in $X$. In other words, for all $x,y \in V(X)$, there exist vertices $p_1, \ldots, p_n \in V(X)$ defining an $(A,B)$-quasigeodesic in $\mathrm{ConeOff}(X, \mathcal{P})$ from $x$ to $y$ and satisfying
$$d_X(x,y)  \geq \frac{1}{A} \sum\limits_{i=1}^{n-1} d_X(p_i,p_{i+1}) -B  .$$
\end{definition}

\noindent
Before stating and proving our criterion, recall that $\mathcal{P}$ is said \emph{locally finite} if there exists a constant $N \geq 0$ such that every vertex belongs to at most $N$ subsets in $\mathcal{P}$. Also, the \emph{growth} of our collection $\mathcal{P}$ is defined as $\gamma_\mathcal{P} := \max_{P \in \mathcal{P}} \max_{x \in P} |B(x,R) \cap P|$. 

\begin{prop}\label{prop:QuasiSyl}
Let $X$ be a graph and $\mathcal{P}$ a locally finite quasi-syllabic collection. The canonical map $\varphi : X \to \mathrm{ConeOff}(X,\mathcal{P})$ is $\gamma_\mathcal{P}(x)^y$-gentle where $\gamma_\mathcal{P}$ denotes the growth function of $\mathcal{P}$ in $X$.
\end{prop}

\begin{proof}
Fix two vertices $p,q \in V(X)$ and two constants $R_1,R_2 \geq 0$. Our goal is to estimate the size of
$$B(p,R_1) \cap \varphi^{-1} (B(q,R_2)).$$
If the intersection is empty, then there is nothing to do, so we assume from now on that our intersection is non-empty. We fix a basepoint $p'$ in this intersection. Notice that
$$B(p,R_1) \cap \varphi^{-1} (B(q,R_2)) \subset B(p',2R_1) \cap \varphi^{-1}(B(q,R_2)),$$
so it suffices to find an upper bound for the size of $B(p',2R_1) \cap \varphi^{-1}(B(q,R_2))$. Let $x$ be an arbitrary vertex in this intersection. Because $\mathcal{P}$ is quasi-syllabic, there exist uniform constants $A>0$, $B \geq 0$ and vertices $p_1, \ldots, p_n \in V(X)$ defining an $(A,B)$-quasigeodesic from $x$ to $p'$ in $\mathrm{ConeOff}(X,\mathcal{P})$ and satisfying
$$d_X(x,p') \geq \frac{1}{A} \sum\limits_{i=1}^{n-1} d_X(p_i,p_{i+1})-B.$$
Notice that, since $d_X(x,p') \leq R_1$, we must have $d_X(p_i,p_{i+1}) \leq A(R_1+B)$ for every $1 \leq i \leq n-1$. Also, notice that $n \leq  2AR_2 +B$. Indeed, 
$$n \leq A \dot{d}(x,p') +B \leq A ( \dot{d}(x,q) + \dot{d}(q,p')) +B \leq 2AR_2 +B$$
where $\dot{d}$ denotes the metric in $\mathrm{ConeOff}(X, \mathcal{P})$. 

\medskip \noindent
For every $1 \leq i \leq n-1$, we know that there exists some $P_i \in \mathcal{P}$ that contains both $p_i$ and $p_{i+1}$. Since $\mathcal{P}$ is locally finite, we know that there is some uniform constant $N$ such that every vertex of $X$ belongs to at most $N$ subgraphs in $\mathcal{P}$. Thus, we have $\leq N$ choices for $P_{n-1}$. Once $P_{n-1}$ is chosen, we know that $p_{n-1}$ belongs to $P_{n-1}$ and lies at distance $\leq A(R_1+B)$ from $p_n=p'$, hence $\leq \gamma_\mathcal{P}(A(R_1+B))$ choices for $p_{n-1}$. Similarly, we then have $\leq N$ choices for $P_{n-2}$ and $\leq \gamma_\mathcal{P}(A(R_1+B))$ choices for $p_{n-2}$. And so on and so forth. We iterate the process $n \leq  2AR_2 +B$ times to reach $p_1=x$. Consequently, we have
$$\leq (N \gamma_\mathcal{P} (A(R_1+B)) )^{ 2AR_2 +B}$$
choices for $x$. This estimate concludes the proof of our proposition.
\end{proof}

\noindent
As an immediate consequence of Proposition~\ref{prop:QuasiSyl}, we obtain the following criterion:

\begin{cor}
Let $X$ be a graph and $\mathcal{P}$ a locally finite quasi-syllabic collection of vertex-sets. If $\mathcal{P}$ has polynomial growth in $X$, then the canonical map $\varphi : X \to \mathrm{ConeOff}(X,\mathcal{P})$ is $x^y$-gentle.
\end{cor}

\section{An osbtruction to polynomial hyperbolicity}

\noindent
In this section, our goal is to prove a criterion, namely Theorem~\ref{thm:NoGentle} below, that allows us to show that some graphs cannot be $\eta$-polynomially hyperbolic for small functions $\eta$. Loosely speaking, our criterion applies to graphs containing quasi-isometrically embedded trees in which any two vertices can be connected by many distinct quasi-geodesics. In order to be more precise, we need to introduce a couple of definitions.

\begin{definition}
Let $X$ be a graph. Two vertices $x,y \in V(X)$ are \emph{exponentially $(a,L)$-connected} if, for every $L \leq R<d(x,y)/2$, we can find $N \geq a^R$ paths $\alpha_1, \ldots, \alpha_N$ of lengths $\leq L d(x,y)$ connecting $x$ and $y$ such that the $\alpha_i \backslash \{x,y\}^{+R}$ are pairwise disjoint.
\end{definition}

\noindent
In our next definition, we denote by $(T_2,o)$ the $2$-regular tree rooted at $o \in V(T_2)$. 

\begin{definition}
A connected graph $X$ \emph{contains an exponentially connected subtree} if there exist $a,L>0$ and a quasi-isometric embedding $\varphi : (T_2,o) \to X$ such that $\varphi(o)$ and $\varphi(x)$ are exponentially $(a,L)$-connected for every $x \in V(T_2)$. 
\end{definition}

\noindent
We are now ready to state the main result of this section.

\begin{thm}\label{thm:NoGentle}
Let $\xi$ be a polynomial satisfying $\xi(x) \to + \infty$ as $x \to + \infty$. A graph $Y$ that contains an exponentially connected subtree cannot be $\xi$-polynomially hyperbolic.
\end{thm}

\noindent
Our proof of Theorem~\ref{thm:NoGentle} actually covers some superpolynomial (but subexponential) maps $\xi$ as well. However, it does not cover every subexponential map, so, for simplicity, we restrict ourselves to polynomial maps. 

\medskip \noindent
The rest of the section is dedicated to the proof of Theorem~\ref{thm:NoGentle}. Our first preliminary lemma shows that, given an $x^{\xi(y)}$-gentle map $\varphi : (T_2,o) \to Z$ and a map $\sigma$ that does not grow too quickly (depending on $\xi$), we find a sequence of vertices $x_n \in V(T_2)$ such that $d(o,x_n)$ grows linearly and $d(\varphi(o),\varphi(x_n))$ grows at least as $\sigma$. 

\begin{lemma}\label{lem:GentleTreeBis}
Let $s \geq 1$ be an integer and $\varphi : (T_2,o) \to Z$ an $x^{y^s}$-gentle map to some metric space $Z$. For every $\sigma(n) =o \left( \left( \frac{n}{\log(n)} \right)^{1/s} \right)$, we can find vertices $x_n \in V(T)$ such that $\left\{ \begin{array}{l} d(o,x_n) = n \\ d(\varphi (o) ,\varphi(x_n)) \geq \sigma(n) \end{array} \right.$ for all but finitely many $n \geq 0$. 
\end{lemma}

\begin{proof}
Assume to the contrary that $d(\varphi(o),\varphi(x_n)) < \sigma(n)$ for every vertex $x_n$ in the sphere $S(o,n)$ and for infinitely many $n \geq 0$. Then there exists a constant $C \geq 0$ such that, for infinitely many $n \geq 0$, we have
$$2^n= |S(o,n)| \leq | B(o,n) \cap \varphi^{-1}(B(\varphi(o),\sigma(n)))| \leq C(Cn)^{(C\sigma(n))^s},$$
hence $(C\sigma(n))^s \geq (n \log(2)- \log(C))/\log(Cn)$, and finally 
$$\sigma(n) \geq \frac{1}{C}  \left( \frac{n \log(2)- \log(C)}{\log(Cn)} \right)^{1/s} \geq \mathrm{cst} \cdot \left( \frac{n}{\log(n)} \right)^{1/s} ,$$ 
which is a contradiction with our assumptions on $\sigma$.
\end{proof}

\noindent
Our second (and last) preliminary lemma quantifies how difficult it is to avoid a ball in a hyperbolic space. 

\begin{lemma}[{\cite[Claim~2 in the proof of Lemma~2.6]{MR3741855}}]\label{lem:HypDiv}
Let $X$ be a hyperbolic space. There exist $\epsilon,s_0>0$ such that, for every $s \geq s_0$, the following holds. Let $x,y \in V(X)$ be two vertices and $B$ a ball of radius $s$ whose centre lies on a geodesic $[x,y]$ at distance $\geq s$ from $\{x,y\}$. Every path connecting $x$ to $y$ but avoiding $B$ has length $\geq (1+ \epsilon)^s$. 
\end{lemma}

\begin{proof}[Proof of Theorem~\ref{thm:NoGentle}.]
We know by assumption that there exist $b,L>0$ and a quasi-isometric embedding $\phi  : (T_2,o) \to Y$ such that $\phi(o)$ and $\phi(x)$ are exponentially $(b,L)$-connected for every $x \in V(T_2)$. For convenience, we identify $T_2$ with its image in $Y$ under $\phi$.

\medskip \noindent
Assume for contradiction that there exists an $x^{y^s}$-gentle map $\varphi : Y \to X$ to some hyperbolic space $X$, which we can assume to be $1$-Lipschitz, for some integer $s \geq 1$. Define three sequences
$$r_n:= \log(n)^2, \ R_n:= \log(n)^{2(s+1)}, \text{ and } \sigma(n):=\log(n)^{2s+3},$$
For convenience, we record the following two elementary observations:

\begin{fact}\label{fact:InequalityOne}
The inequality $R_n < \frac{n}{2}$ holds for all but finitely many $n \geq 2$. 
\end{fact}

\begin{fact}\label{fact:InequalityTwo}
We have $\sigma(n)= o \left( \left( \frac{n }{ \log(n)} \right)^{1/s} \right).$
\end{fact}

\noindent
By applying Lemma~\ref{lem:GentleTreeBis} thanks to Fact~\ref{fact:InequalityTwo}, we know that, for sufficiently large values of $n$, there exist $x_n \in T_2$ such that $d(o,x_n)= n$ and $d(\varphi(o), \varphi(x_n)) \geq \sigma(n)$. And, because our subtree is exponentially connected and thanks to Fact~\ref{fact:InequalityOne}, we can find $N \geq b^{R_n}$ paths $\alpha_1, \ldots, \alpha_N$ of lengths $\leq Ln$ connecting $o$ to $x_n$ such that $\alpha_i \backslash \{o,x_n\}^{+R_n}$ are pairwise disjoint. 

\medskip \noindent
If $m_n$ denotes a middle point of $\varphi(o)$ and $\varphi(x_n)$ in $X$, we claim that there exists some $1 \leq i \leq N$ such that $\alpha_i$ is disjoint from $B(m_n,r_n)$.

\medskip \noindent
Otherwise, for every $1 \leq i \leq N$, we can find a point $p_i \in \alpha_i$ that belongs to $\varphi^{-1}(B(m_n,r_n))$. If $p_i \in B(o,R_n)$, then
$$d(\varphi(o),m_n) \leq d(\varphi(o), \varphi(p_i))+ d(\varphi(p_i),m_n) \leq d(o,p_i) + r_n \leq R_n+r_n,$$
hence
$$\frac{\sigma(n) }{2} \leq \frac{d(\varphi(o),\varphi(x_n))}{2} = d(\varphi(o),m_n) \leq R_n+r_n,$$
which contradicts the definition of $\sigma$. Therefore, $p_i$ does not belong to $B(o,R_n)$. By symmetry, we show that $p_i$ does not belong to $B(x_n,R_n)$ either. In other words, $p_i$ belongs to $\alpha_i \backslash \{o,x_n\}^{+R_n}$. We deduce that the $p_i$ are pairwise distinct, and then that
$$C(Cn)^{(C r_n)^s}\geq |B(o,n) \cap \varphi^{-1}(B(m_n,r_n))| \geq N \geq b^{R_n}$$
hence
$$R_n \leq \mathrm{cst} \cdot r_n^s \log(n) = \mathrm{cst} \cdot \log(n)^{2s+1}$$
for some uniform constant $C>0$ and for infinitely many values of $n$. This contradicts the definition of $R_n$. 

\medskip \noindent
Thus, for each large value of $n$, we can find some $1 \leq i \leq N$ such that $\alpha_i$ is disjoint from $B(m_n,r_n)$. We deduce from Lemma~\ref{lem:HypDiv} that, for large values of $n$ and for some uniform constant $\epsilon>0$, we have
$$Ln \geq \mathrm{length}(\alpha_i) \geq \mathrm{length}(\varphi(\alpha_i)) \geq (1+\epsilon)^{r_n},$$
hence $r_n \leq \mathrm{cst} \cdot \log(n)$, which contradicts the definition of $r_n$. 
\end{proof}

\section{Application to lamplighters}

\noindent
In this section, we are interested in the large scale geometry of the lamplighter graph $\mathrm{Lamp}$
\begin{itemize}
	\item whose vertices are the pairs $(S,p)$ where $S \subset \mathbb{Z}$ is a finite subset and where $p \in \mathbb{Z}$ is an integer;
	\item and whose edges connect $(S_1,p_1)$ to $(S_2,p_2)$ whenever either $S_1=S_2$ and $|p_1-p_2|=1$ or $S_2= S_1 \triangle \{p_1=p_2\}$.
\end{itemize}
Visually, we image a bi-infinite street with lamps labelled by $\mathbb{Z}$ along which travels a lamplighter. In a pair $(S,p)$, $S$ is interpreted as the set of lamps the lamplighter switched on and $p$ as the position of the lamplighter. Then, moving in $\mathrm{Lamp}$ from vertices to adjacent vertices amounts to moving the lamplighter along the street and to switch on or off the lamp where the lamplighter is.

\medskip \noindent
More precisely, we are interested in the smallest $\eta : [0,+ \infty) \to [0,+ \infty)$ for which the lamplighter graph is $\eta$-polynomially hyperbolic. According to Example~\ref{ex:ExpPolynomial}, we know that $\eta$ is at most exponential. The main result of this section is that $\eta$ must be superpolynomial:

\begin{thm}\label{thm:GentleLamp}
Let $\xi$ be a polynomial satisfying $\xi(x) \to + \infty$ as $x \to + \infty$. The lamplighter graph cannot be $\xi$-polynomially hyperbolic.
\end{thm}

\noindent
The strategy is of course to apply Theorem~\ref{thm:NoGentle}. We start by proving the following observation:

\begin{lemma}\label{lem:ExpConnectedLamp}
For every vertex $(S,p)$ of $\mathrm{Lamp}$ with $S \subset [0,p]$, the vertices $(\emptyset, 0)$ and $(S,p)$ are exponentially $(\sqrt[4]{2},6)$-connected.
\end{lemma}

\begin{proof}
Let $6 \leq R<d((c,p),(\emptyset, 0))/2$. For a non-empty subset $A \subset [-R/2,0)$ and a non-empty subset $B \subset (p,p+R/2]$, we define the path $\pi(A,B)$ as the following concatenation of geodesics:
\begin{center}
\includegraphics[width=0.7\linewidth]{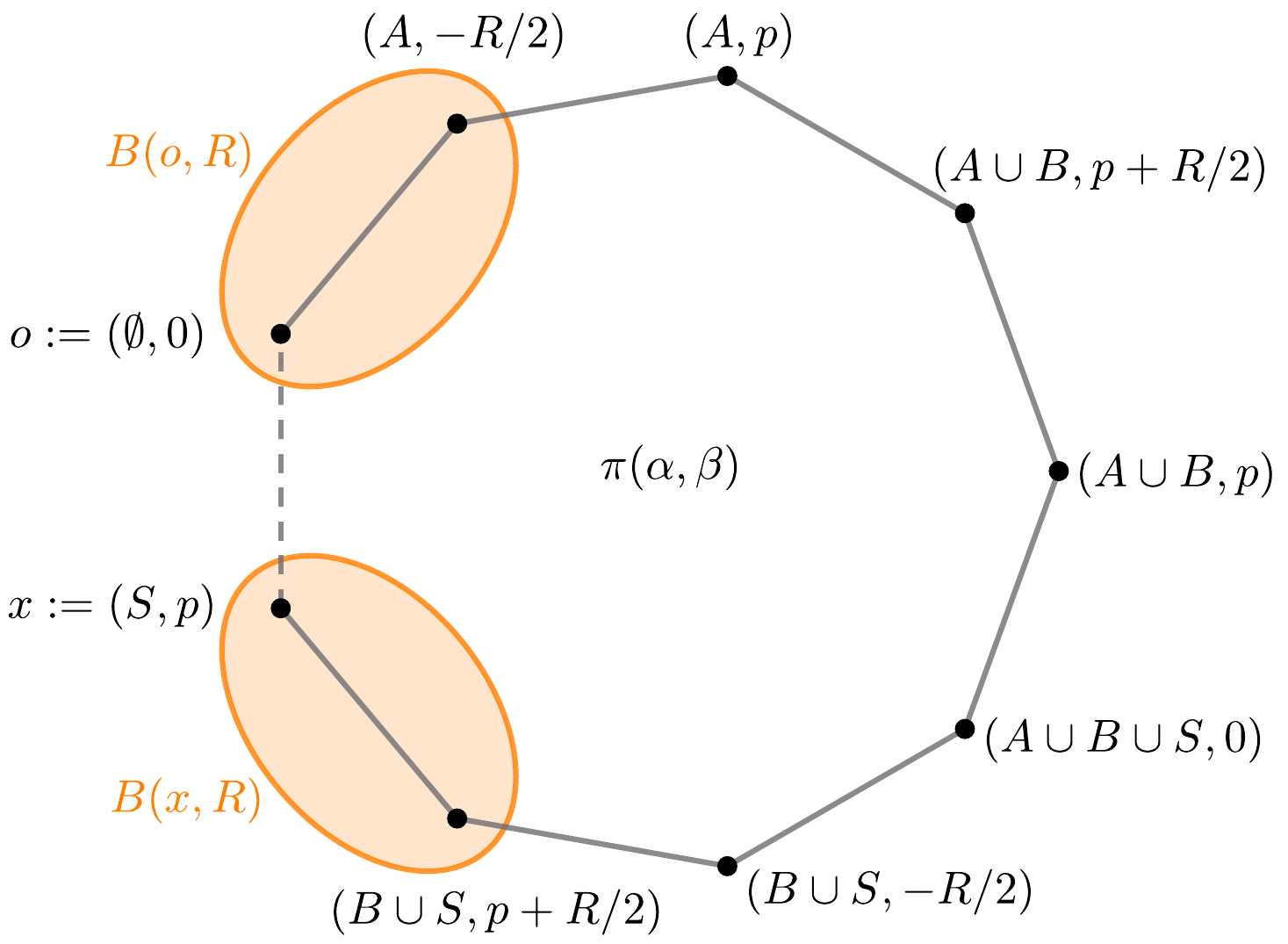}
\end{center}

\noindent
Notice that $\pi(A,B)$ has length $\leq 6 d(o,x)$; and that, given two other non-empty subsets $A' \subset [-R/2,0)$ and $B' \subset (p,p+R/2]$, $\pi(A,B)$ is disjoint from $\pi(A',B')$ outside the balls $B(o,R)$ and $B(x,R)$ whenever $A \neq A'$ and $B \neq B'$. Since we have 
$$2^{\lfloor R/2 \rfloor}-1 \geq \frac{7}{8} 2^{\lfloor R/2 \rfloor} \geq \frac{7}{16} \sqrt{2}^R \geq \frac{7}{16} \left( \frac{\sqrt{2}}{\sqrt[4]{2}} \right)^6 \sqrt[4]{2}^R \geq \sqrt[4]{2}^R$$ 
choices for $A$ and similarly $\geq \sqrt[4]{2}^R$ choices for $B$, the desired conclusion follows. 
\end{proof}

\begin{proof}[Proof of Theorem~\ref{thm:GentleLamp}.]
According to Theorem~\ref{thm:NoGentle}, it suffices to verify that $\mathrm{Lamp}$ contains an exponentially connected subtree. Consider the map 
$$\Phi : \left\{ \begin{array}{ccc} \{0,1\}^\ast & \to & \mathrm{Lamp} \\ x_1 \cdots x_n & \mapsto &  (\{i \geq 1 \mid x_i =1\}, n) \end{array} \right.$$
where $\{0,1\}^\ast$ denotes the set of finite strings of $0$s and $1$s. Notice that the rooted tree $(T_2,o)$ can be described as the graph whose vertex-set is $\{0,1\}^\ast$ and whose edges connect two strings whenever one can be obtained from the other by adding a single digit (to the right). Thus, $\Phi$ naturally induces a map $(T_2,o) \to \mathrm{Lamp}$. This map is clearly an isometric embedding, and its image is $(\sqrt[4]{2},6)$-exponentially connected as a consequence of Lemma~\ref{lem:ExpConnectedLamp}. 
\end{proof}

\noindent
Since $\mathrm{Lamp}$ quasi-isometrically embeds into a product of two $3$-regular trees $T_3 \times T_3$ (see e.g.\ \cite[Corollary~10]{MR3079268}), Theorem~\ref{thm:GentleLamp} (combined with Lemma~\ref{lem:GentleStableQI}) immediately implies that:

\begin{cor}
Let $\xi$ be a polynomial satisfying $\xi(x) \to + \infty$ as $x \to + \infty$. A product $T_3 \times T_3$ of two $3$-regular trees cannot be $\xi$-polynomially hyperbolic.
\end{cor}

\section{Application to cocompact special groups}

\noindent
In this section, our goal is to prove Theorem~\ref{thm:IntroSpecial} from the introduction, which we restate below for the reader's convenience:

\begin{thm}\label{thm:SpecialLamp}
Let $G$ be a cocompact special group. The following are equivalent:
\begin{itemize}
	\item $G$ contains $\mathbb{F}_2 \times \mathbb{F}_2$ as a subgroup;
	\item $G$ contains $\mathbb{F}_2 \times \mathbb{F}_2$ as an undistorted subgroup;
	\item $\mathbb{F}_2 \times \mathbb{F}_2$ quasi-isometrically embeds into $G$;
	\item $\mathrm{Lamp}(\mathbb{Z})$ quasi-isometrically embeds into $G$;
	\item $G$ is not $\mathrm{pol}$-polynomially hyperbolic;
	\item $G$ is not $\mathrm{lin}$-polynomially hyperbolic.
\end{itemize}
\end{thm}

\noindent
In Sections~\ref{section:BasicMedian} and~\ref{section:IntroSpecial}, we recall some basic definitions and properties related to median geometry and cocompact special groups, and we record a few preliminary lemmas. The proof of Theorem~\ref{thm:SpecialLamp} is contained in Section~\ref{section:ProofSpecial}

\subsection{Preliminaries on median geometry}\label{section:BasicMedian}

\noindent
We start by recalling basic definitions and properties related to median graphs. 

\begin{definition}
A connected graph $X$ is \emph{median} if, for all vertices $x_1,x_2,x_3 \in V(X)$, there exists a unique vertex $m \in V(X)$ satisfying
$$d(x_i,x_j)= d(x_i,m)+ d(m,x_j) \text{ for all } i \neq j.$$
\end{definition}

\noindent
Hyperplanes, which we now define, provide a fundamental tool in order to understand the geometry of median graphs.

\begin{definition}
Let $X$ be a median graph. A(n \emph{oriented}) \emph{hyperplane} is an equivalence class of (oriented) edges with respect to the reflexive-transitive closure of the relation that identifies two (oriented) edges whenever they are opposite in some induced $4$-cycle. Two hyperplanes are \emph{transverse} (resp.\ \emph{tangent}) whenever they contain two edges with a common endpoint that span (resp.\ that do not span) an induced $4$-cycle. 
\end{definition}
\begin{figure}
\begin{center}
\includegraphics[width=0.4\linewidth]{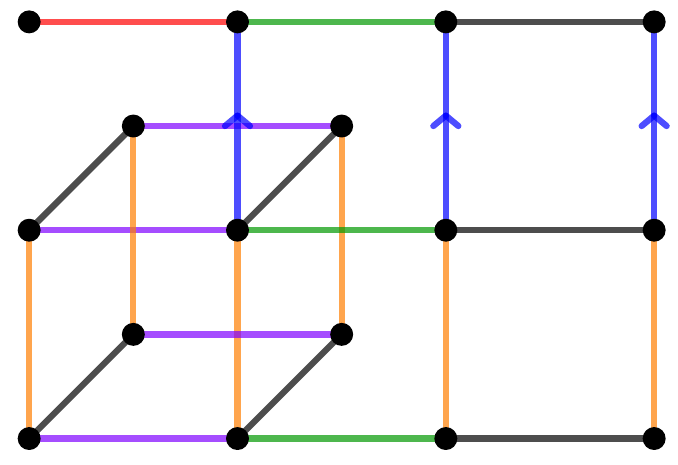}
\caption{Examples of (oriented) hyperplanes in a median graph. The green hyperplane is tangent to the red and purple hyperplanes, and transverse to the orange hyperplane.}
\label{Hyp}
\end{center}
\end{figure}

\noindent
See Figure~\ref{Hyp}. Roughly speaking, the geometry of a median graph is encoded in the combinatorics of its hyperplanes. This idea is partly justified by the following statement:

\begin{thm}\label{thm:BigMedian}
Let $X$ be a median graph. 
\begin{itemize}
	\item For every hyperplane $J$, the graph $X \backslash \backslash J$ obtained from $X$ by removing the edges in $J$ has exactly two connected components, referred to as \emph{halfspaces}. 
	\item Halfspaces are convex.
	\item A path is a geodesic if and only if it crosses every hyperplane at most once. 
	\item The distance between two vertices coincide with the number of hyperplanes separating them.
\end{itemize}
\end{thm}

\noindent
We now turn to some preliminary lemmas that will be useful later. First, the following definition will be needed in order to state Proposition~\ref{prop:SyllabicMedian} in Section~\ref{section:ProofSpecial}. 

\begin{definition}
Let $X$ be a median graph. Two pairs of vertices $(a,b),(x,y) \in V(X)^2$ are \emph{parallel} if every halfspace contains $a$ but not $b$ if and only if it contains $x$ but not $y$. 
\end{definition}

\noindent
Finally, we record two elementary observations related to parallel pairs of vertices. In our first lemma, we denote by $\mathcal{H}(\cdot | \cdot)$ the collection of all the hyperplanes separating two given vertices. 

\begin{lemma}\label{lem:ParallelPairsTrans}
Let $X$ be a median graph and $(a,b),(x,y) \in V(X)^2$ two parallel pairs of vertices. The collections of hyperplanes $\mathcal{H}(x|a)$ and $\mathcal{H}(x|y)$ are transverse.
\end{lemma}

\begin{proof}
Let $J$ be a hyperplane separating $x$ and $a$. If $J$ separates $a$ and $b$, then it delimits a halfspace containing $a$ but neither $b$ nor $x$, which is impossible. Similarly, $J$ cannot separate $x$ and $y$. As a consequence, $J$ must separate $b$ and $y$. Thus, every hyperplane separating $x$ and $a$ must separate $\{a,b\}$ and $\{x,y\}$. On the other hand, it is clear that every hyperplane separating $x$ and $y$ separates $\{a,x\}$ and $\{b,y\}$. We conclude that every hyperplane separating $x$ and $a$ must be transverse to every hyperplane separating $x$ and~$y$. 
\end{proof}

\begin{lemma}\label{lem:ParallelPairsHyp}
Let $X$ be a median graph and $(a,b),(x,y) \in V(X)^2$ two parallel pairs of vertices. For every geodesic $\alpha$ connecting $a$ to $b$, there exists a geodesic $\xi$ connecting $x$ to $y$ such that $\alpha$ and $\xi$ cross the same hyperplanes in the same order. 
\end{lemma}

\begin{proof}
Let $a'$ denote the neighbour of $a$ along $\gamma$ and consider the median point $x'$ of $x,y,a'$. We claim that $x'$ is a neighbour of $x$ and that the unique hyperplane separating $x$ and $x'$ coincides with the hyperplane separating $a$ and $a'$.

\medskip \noindent
Let $J$ be a hyperplane separating $x$ and $x'$. Since $J$ also separates $x$ and $y$, necessarily $a$ and $b$ are separated by $J$. Moreover, the halfspace $J^+$ delimited by $J$ and containing $x$ must contain $a$ as well. If $a'$ belongs to $J^+$, then the convexity of $J^+$ would imply that $x' \in J^+$, contradicting the assumption that $J$ separates $x$ and $x'$. Hence $a' \notin J^+$, which implies that $J$ separates $a$ and $a'$. Conversely, assume that $J$ is the hyperplane separating $a$ and $a'$. Since $J$ separates $a$ and $b$, necessarily $x$ and $y$ are separated by $J$. Since $a'$ and $y$ belong to the same halfspace delimited by $J$, necessarily the median point $x'$ must belong to the same halfspace by convexity. This implies that $J$ separates $x$ and $x'$. This concludes the proof of our claim.

\medskip \noindent
Thus, we have found a neighbour $x'$ of $x$ that belongs to a geodesic connecting $x$ to $y$ such that the edge $\{x,x'\}$ belongs to the same hyperplane as the first edge of $\alpha$. The edge $\{x,x'\}$ will be the first edge of the geodesic $\xi$ we are looking for. In order to iterate the process, it suffices to verify that the new pairs $(a',b)$ and $(x',y)$ are parallel.

\medskip \noindent
Let $J$ be a hyperplane delimiting a halfspace $J^+$ containing $a'$ but not $b$. Since $J^+$ contains $a$ but not $b$, necessarily it contains $x$ but not $y$. Because $J \notin \mathcal{H}(a|a') = \mathcal{H}(x|x')$, it follows that $x' \in J^+$. Thus, $J^+$ contains $x'$ but not $y$. Conversely, assume that $J$ is a hyperplane delimiting a halfspace $J^+$ containing $x'$ but not $y$. Since $J$ contains $x$ but not $y$, necessarily it contains $a$ but not $b$. Again, because $J \notin \mathcal{H}(x|x')= \mathcal{H}(a|a')$, it follows that $a' \in J^+$. Thus, $J^+$ contains $a'$ but not $b$, as desired. 
\end{proof}

\subsection{Combinatorics of special cube complexes}\label{section:IntroSpecial}

\noindent
In this section, we recall basic definitions related to special cube complexes, as introduced in \cite{MR2377497}, and we record the formalism introduced in \cite{SpecialRH} that we will use in Section~\ref{section:ProofSpecial} in order to prove Theorem~\ref{thm:IntroSpecial}. 

\begin{definition}
A cube complex is \emph{nonpositively curved} if the links of its vertices are flag simplicial complexes. 
\end{definition}

\noindent
As shown in \cite{MR919829} (see also \cite{MR1744486}), a cube complex is CAT(0) if and only if is simply connected and nonpositively curved. Then, it follows from \cite{MR1663779,Roller,MR1748966} that the one-skeleta of universal covers of nonpositively curved cube complexes coincide with median graphs. We define (oriented) hyperplanes in nonpositively curved cube complexes as images of hyperplanes in the universal covers. 

\begin{definition}
Let $J$ be a hyperplane with a fixed orientation $\vec{J}$ in a nonpositively curved cube complex. We say that $J$ is:
\begin{itemize}
	\item \textit{2-sided} if $\vec{J} \neq \vec{J}^{-1}$, where $\vec{J}^{-1}$ denote the opposite orientation of $\vec{J}$;
	\item \textit{self-intersecting} if there exist two edges dual to $J$ that are non-parallel sides of some square;
	\item \textit{self-osculating} if there exist two oriented edges dual to $\vec{J}$ with the same initial points or the same terminal points, but which do not belong to a same square.
\end{itemize}
Moreover, if $H$ is another hyperplane, then $J$ and $H$ are:
\begin{itemize}
	\item \textit{transverse} if there exist two edges dual to $J$ and $H$ respectively that are non-parallel sides of some square;
	\item \textit{inter-osculating} if they are transverse, and if there exist two edges dual to $J$ and $H$ respectively with one common endpoint, but which do not belong to a same square.
\end{itemize}
\end{definition}
\noindent
Sometimes, one refers 1-sided, self-intersecting, self-osculating, and inter-osculating hyperplanes as \textit{pathological configurations of hyperplanes}. The last three configurations are illustrated on Figure \ref{figure17}.
\begin{figure}
\begin{center}
\includegraphics[width=\linewidth]{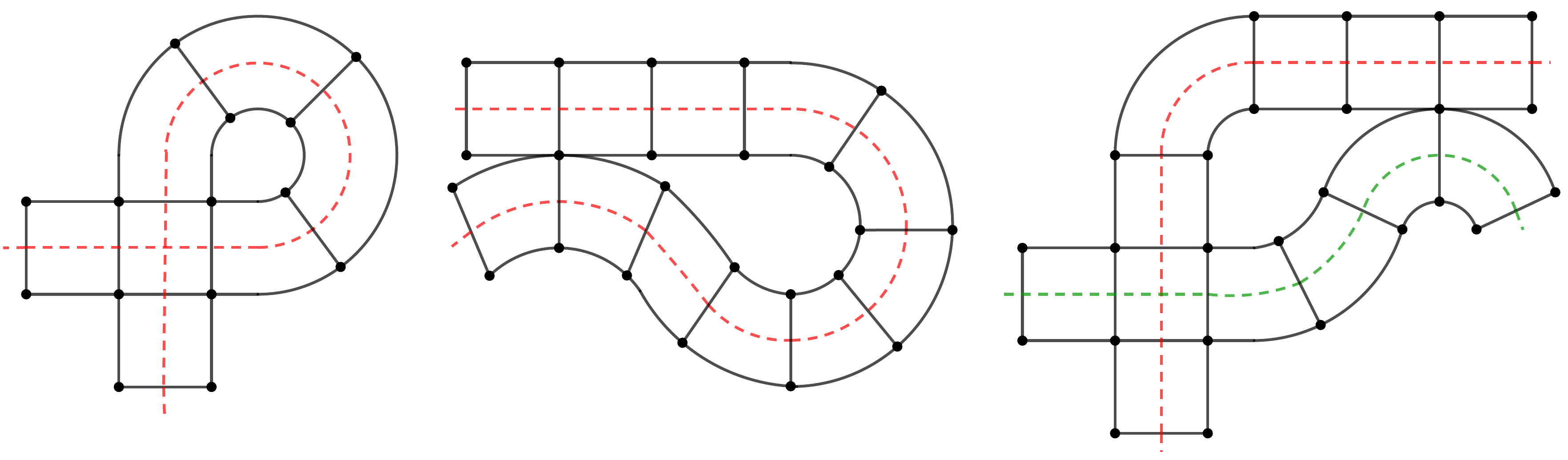}
\caption{From left to right: self-intersection, self-osculation, inter-osculation.}
\label{figure17}
\end{center}
\end{figure}

\begin{definition}
A \emph{special cube complex} is a nonpositively curved cube complex whose hyperplanes are two-sided and that does not contain self-intersecting, self-osculating, nor inter-osculating hyperplanes. A group that can be described as the fundamental group of a (compact) special cube complex is \emph{(cocompact) special}.
\end{definition}

\noindent
Let $Q$ be a cube complex (not necessarily nonpositively-curved). For us, a \emph{(combinatorial) path} in $Q$ is the data of a sequence of successively adjacent edges. What we want to understand is when two such paths are homotopic (with fixed endpoints). For this purpose, we need to introduce the following elementary transformations. One says that:  
\begin{itemize}
	\item a path $\gamma \subset Q$ contains a \emph{backtrack} if the word of oriented edges read along $\gamma$ contains a subword $ee^{-1}$ for some oriented edge $e$;
	\item a path $\gamma' \subset Q$ is obtained from another path $\gamma \subset Q$ by \emph{flipping a square} if the word of oriented edges read along $\gamma'$ can be obtained from the corresponding word of $\gamma$ by replacing a subword $e_1e_2$ with $e_2'e_1'$ where $e_1',e_2'$ are opposite oriented edges of $e_1,e_2$ respectively in some square of $Q$. 
\end{itemize}
We claim that these elementary transformations are sufficient to determine whether or not two paths are homotopic. More precisely:

\begin{prop}\label{prop:cubehomotopy}
Let $Q$ be a cube complex and $\gamma,\gamma' \subset Q$ two paths with the same endpoints. The paths $\gamma,\gamma'$ are homotopic (with fixed endpoints) if and only if $\gamma'$ can be obtained from $\gamma$ by removing or adding backtracks and flipping squares. 
\end{prop}

\noindent
This statement follows from the fact that flipping squares provide the relations of the fundamental groupoid of $Q$; see \cite[Statement 9.1.6]{BrownGroupoidTopology} for more details. 

\medskip \noindent
In the rest of this subsection, we describe (a particular case of) the formalism introduced in \cite{SpecialRH}. 

\medskip \noindent
From now on, we fix a (not necessarily compact) special cube complex $Q$. The \emph{crossing graph}\footnote{The graph $\Delta Q$ will define the right-angled Artin group in which the fundamental group of $Q$ will embed. This is the same graph used in \cite{MR2377497}, where no name is assigned to it. Here, we use the suggestive terminology \emph{crossing graph}, but it should not be confused with related but different graphs introduced in various contexts, e.g.\ \cite{MR1153934, MR1379364, Roller, MR1920184, MR3217625}.} $\Delta Q$ is the graph whose vertices are the hyperplanes of $Q$ and whose edges link two transverse hyperplane. 

\medskip \noindent
Given a vertex $o \in Q^{(0)}$, a word $w=J_1 \cdots J_r$ of oriented hyperplanes $J_1, \ldots, J_r$ is \emph{$o$-legal} if there exists a path $\gamma$ in $Q$ starting from $o$ such that the oriented hyperplanes it crosses are successively $J_1, \ldots, J_r$. We say that the path $\gamma$ \emph{represents} the word $w$. 

\begin{fact}\label{fact:uniquepath}\emph{\cite[Fact 3.3]{SpecialRH}}
An $o$-legal word is represented by a unique path in $Q$. 
\end{fact}

\noindent
The previous fact allows us to define the \emph{terminus} of an $o$-legal word $w$, denoted by $t(w)$, as the ending point of the unique path representing $w$. 

\medskip \noindent
Set $\mathcal{L}(Q) = \{ x \text{-legal words} \mid x \in Q^{(0)}\}$ the set of all legal words. (If $x_1,x_2 \in Q^{(0)}$ are two distinct points, we consider the empty $x_1$-legal and $x_2$-legal words as distinct.) We consider the equivalence relation $\sim$ on $\mathcal{L}(Q)$ generated by the following transformations:
\begin{description}
	\item[(cancellation)] if a legal word contains $JJ^{-1}$ or $J^{-1}J$, remove this subword;
	\item[(insertion)] given an oriented hyperplane $J$, insert $JJ^{-1}$ or $J^{-1}J$ as a subword of a legal word;
	\item[(commutation)] if a legal word contains $J_1J_2$ where $J_1,J_2$ are two transverse oriented hyperplanes, replace this subword with $J_2 J_1$. 
\end{description}
So two $x$-legal words $w_1,w_2$ are equivalent with respect to $\sim$ if there exists a sequence of $x$-legal words 
$$m_1=w_1, \ m_2, \ldots, m_{r-1}, \ m_r=w_2$$ 
such that $m_{i+1}$ is obtained from $m_i$ by a cancellation, an insertion, or a commutation for every $1 \leq i \leq r-1$. Define $\mathcal{D}(Q)= \mathcal{L}(Q)/ \sim$ as a set of \emph{diagrams}. The following observation allows us (in particular) to define the \emph{terminus} of a diagram as the terminus of one of the legal words representing it.

\begin{fact}\label{fact:transformations}\emph{\cite[Fact 3.4]{SpecialRH}}
Let $w'$ be an $o$-legal word obtained from another $o$-legal word $w$ by a cancellation / an insertion / a commutation. If $\gamma',\gamma$ are paths representing $w',w$ respectively, then $\gamma'$ is obtained from $\gamma$ by removing a backtrack / adding a backtrack / flipping a square. 
\end{fact}

\noindent
In the sequel, an \emph{$(x,y)$-diagram} will refer to a diagram represented by an $x$-legal word with terminus $y$, or just an \emph{$(x,\ast)$-diagram} if we do not want to specify its terminus. A diagram which is an $(x,x)$-diagram for some $x \in X$ is \emph{spherical}. 

\medskip \noindent
If $w$ is an $o$-legal word and $w'$ a $t(w)$-legal word, we define the \emph{concatenation} $w \cdot w'$ as the word $ww'$, which is $o$-legal since it is represented by the concatenation $\gamma \gamma'$ where $\gamma, \gamma'$ represent respectively $w,w'$. Because we have the natural identifications
\begin{table}[h]
	\centering
		\begin{tabular}{cccc}
			$\mathcal{L}(Q)$ & $\leftrightarrow$ & paths in $Q$ & (Fact \ref{fact:uniquepath}) \\ 
			$\sim$ & $\leftrightarrow$ & homotopy with fixed endpoints & (Fact \ref{fact:transformations}, Proposition \ref{prop:cubehomotopy}) 
		\end{tabular}
\end{table}

\noindent
it follows that the concatenation in $\mathcal{L}(Q)$ induces a well-defined operation in $\mathcal{D}(Q)$, making $\mathcal{D}(Q)$ isomorphic to the fundamental groupoid of $Q$. As a consequence, if we denote by $M(Q)$ the Cayley graph of the groupoid $\mathcal{D}(Q)$ with respect to the generating set given by the oriented hyperplanes, and, for every $x \in Q^{(0)}$, $M(Q,x)$ the connected component of $M(Q)$ containing the trivial path $\epsilon(x)$ based at $x$, and $\mathcal{D}(Q,x)$ the vertex-group of $\mathcal{D}(Q)$ based at $\epsilon(x)$, then the previous identifications induce the identifications
\begin{table}[h]
	\centering
		\begin{tabular}{ccc}
			$\mathcal{D}(Q)$ & $\leftrightarrow$ & fundamental groupoid of $Q$  \\ 
			$\mathcal{D}(Q,x)$ & $\leftrightarrow$ & $\pi_1(X,x)$  \\ 
			$M(Q,x)$ & $\leftrightarrow$ & universal cover $\widetilde{Q}^{(1)}$ with a specified basepoint  \\ 
		\end{tabular}
\end{table}

\noindent
More explicitly, $\mathcal{D}(Q,x)$ is the group of $(x,x)$-diagrams endowed with the concatenation, and $M(Q,x)$ is the graph whose vertices are the $(x,\ast)$-diagrams and whose edges link two diagrams $w_1$ and $w_2$ if there exists some oriented hyperplane $J$ such that $w_2=w_1J$. 

\medskip \noindent 
A more precise description of the identification between $M(Q,x)$ and $\widetilde{Q}^{(1)}$ is the following:
\begin{table}[h!]
	\centering
		\begin{tabular}{l}
			\hspace{5cm} $M(Q,x) \longleftrightarrow \left( \widetilde{Q}, \widetilde{x} \right)$ \\ \\ 
			\begin{tabular}{c} $(x,\ast)$-diagram represented \\ by an $x$-legal word $w$ \end{tabular}  $\mapsto$ \begin{tabular}{c} path $\gamma \subset Q$ \\ representing $w$ \end{tabular}  $\mapsto$  \begin{tabular}{c} lift $\widetilde{\gamma} \subset \widetilde{Q}$ of $\gamma$ \\ starting from $\widetilde{x}$ \end{tabular} $\mapsto$ \begin{tabular}{c} ending \\ point of $\widetilde{\gamma}$ \end{tabular} \\ \\
			\begin{tabular}{c} $(x,\ast)$-diagram represented by the \\ $x$-legal word corresponding to $\gamma$ \end{tabular} $\mapsfrom$ \begin{tabular}{c} image \\ $\gamma \subset Q$ of $\gamma$ \end{tabular} $\mapsfrom$ \begin{tabular}{c} path $\widetilde{\gamma} \subset \widetilde{Q}$ \\ from $\widetilde{x}$ to $y$ \end{tabular} $\mapsfrom$ $y$
		\end{tabular}
\end{table}

\noindent
A diagram may be represented by several legal words. Such a word is \emph{reduced} if it has minimal length, i.e.\ it cannot be shortened by applying a sequence of cancellations, insertions and commutation. It is worth noticing that, in general, a diagram is not represented by a unique reduced legal word, but two such reduced words differ only by some commutations. (For instance, consider the homotopically trivial loop defined by the paths representing two of our reduced legal words, consider a disc diagram of minimal area bounded by this loop, and follow the proof of \cite[Theorem 4.6]{MR1347406}. Alternatively, use the embedding constructed in \cite[Section 4.1]{SpecialRH} (which does not use the present discussion) and conclude by applying the analogous statement which holds in right-angled Artin groups.) As a consequence, we can define the \emph{length} of a diagram as the length of any reduced legal word representing it, and its \emph{support} as the set of hyperplanes of $Q$ given by the letters of a reduced representative. It is worth noticing that our length coincides with the length which is associated to the generating set given by the oriented hyperplanes in the groupoid $\mathcal{D}(Q)$. The next lemma follows from this observation.

\begin{lemma}\label{lem:GeodesicsSpecial}
Let $D_1,D_2 \in M(Q,x)$ be two $(x,\ast)$-diagrams. If $J_1 \cdots J_n$ is a reduced legal word representing $D_1^{-1}D_2$, then 
$$D_1, \ D_1J_1, \ D_1J_1J_2, \ldots, \ D_1J_1 \cdots J_n$$
is a geodesic from $D_1$ to $D_2$ in $M(Q,x)$. Conversely, any geodesic between $D_1$ and $D_2$ arises in this way.
\end{lemma}

\noindent
Essentially by construction, one has:

\begin{fact}\label{fact:diagvsRAAG}
\emph{\cite[Fact 4.3]{SpecialRH}}
Let $Q$ be a special cube complex and $o \in Q^{(0)}$ a basepoint. Two $o$-legal words of oriented hyperplanes are equal in $\mathcal{D}(Q,o)$ if and only if they are equal in the right-angled Artin group $A(\Delta Q)$. 
\end{fact}

\noindent
As a consequence, a reduced legal word represents a trivial element in $\mathcal{D}(Q,o)$ if and only if it is empty. 

\medskip \noindent
We conclude this section by stating and proving a few preliminary lemmas that will be useful in Section~\ref{section:ProofSpecial}. First, we describe induced $4$-cycles in our median graph $M(Q,o)$.

\begin{lemma}\label{lem:SquareM}
Every induced $4$-cycle in $M(Q,o)$ is of the form $\{w,wa,wb,wab\}$ where $w$ is an $(o,\ast)$-legal word and $a,b$ are two transverse oriented hyperplanes of $Q$. 
\end{lemma}

\begin{proof}
Let $S$ be an induced square in $M(Q,o)$. Fix an $(o,\ast)$-legal word $w$ representing some vertex of $S$. Its two neighbours in $S$ can be written as $wa$ and $wb$ for some oriented hyperplanes $a$ and $b$ of $Q$. The fourth vertex of $S$ can be written both as $wax$ and as $wby$ for some oriented hyperplanes $x,y$ of $Q$. From the equality $wax=wby$ in $\mathcal{D}(Q)$, we deduce that $a^{-1}b=xy^{-1}$. Because $wa$ and $wb$ represent distinct vertices of $S$, necessarily $a \neq b$, so the word $a^{-1}b$ is reduced. Since two reduced word representing the same element of $\mathcal{D}(Q)$ only differ by some commutations, only two cases may happen: either $x=a^{-1}$ and $y=b^{-1}$, which is not possible since $wax$ and $wa$ represent distinct vertices of $S$; or $a,b$ are transverse hyperplanes of $Q$ and $x=b$, $y=a$. 
\end{proof}

\noindent
An immediate consequence of Lemma~\ref{lem:SquareM} is that two (oriented) edges that are opposite in some $4$-cycle must be labelled by the same (oriented) hyperplane of $Q$. This implies that:

\begin{cor}
Two oriented edges in $M(Q,o)$ that belong to the same hyperplane are labelled by the same oriented hyperplane of $Q$. 
\end{cor}

\noindent
This allows us to naturally label the (oriented) hyperplanes of $M(Q,o)$ with (oriented) hyperplanes of $Q$. Then, another immediate consequence of Lemma~\ref{lem:SquareM} is that:

\begin{cor}\label{cor:LabelTransHyp}
Two transverse hyperplanes in $M(Q,o)$ are labelled by transverse hyperplanes of $Q$.
\end{cor}

\noindent
Finally, we record our last preliminary lemma:

\begin{lemma}\label{lem:TransverseLegal}
Let $p$ and $q$ be two vertices of $M(Q,o)$ and let $\Lambda$ be a collection of hyperplanes of $Q$ transverse to $\mathrm{supp}(p^{-1}q)$. A word of hyperplanes in $\Lambda$ is $(t(p),t(p))$-legal if and only if it is $(t(q),t(q))$-legal.
\end{lemma}

\begin{proof}
Let $w$ be a $(t(p),t(p))$-legal word of hyperplanes in $\Lambda$. In other words, there exists a loop $\gamma$ in $Q$ based at $t(p)$ such that the word $w$ records the oriented hyperplanes (of $\Lambda$) successively crossed by $\gamma$. Fix a path $\pi$ in $Q$ connecting $t(p)$ to $t(q)$ that crosses only hyperplanes in $\mathrm{supp}(p^{-1}q)$. Denote by $\gamma_1, \ldots, \gamma_r$ and $\pi_1, \ldots, \pi_s$ the successive edges of $\gamma$ and $\pi$ respectively. 
\begin{center}
\includegraphics[width=0.5\linewidth]{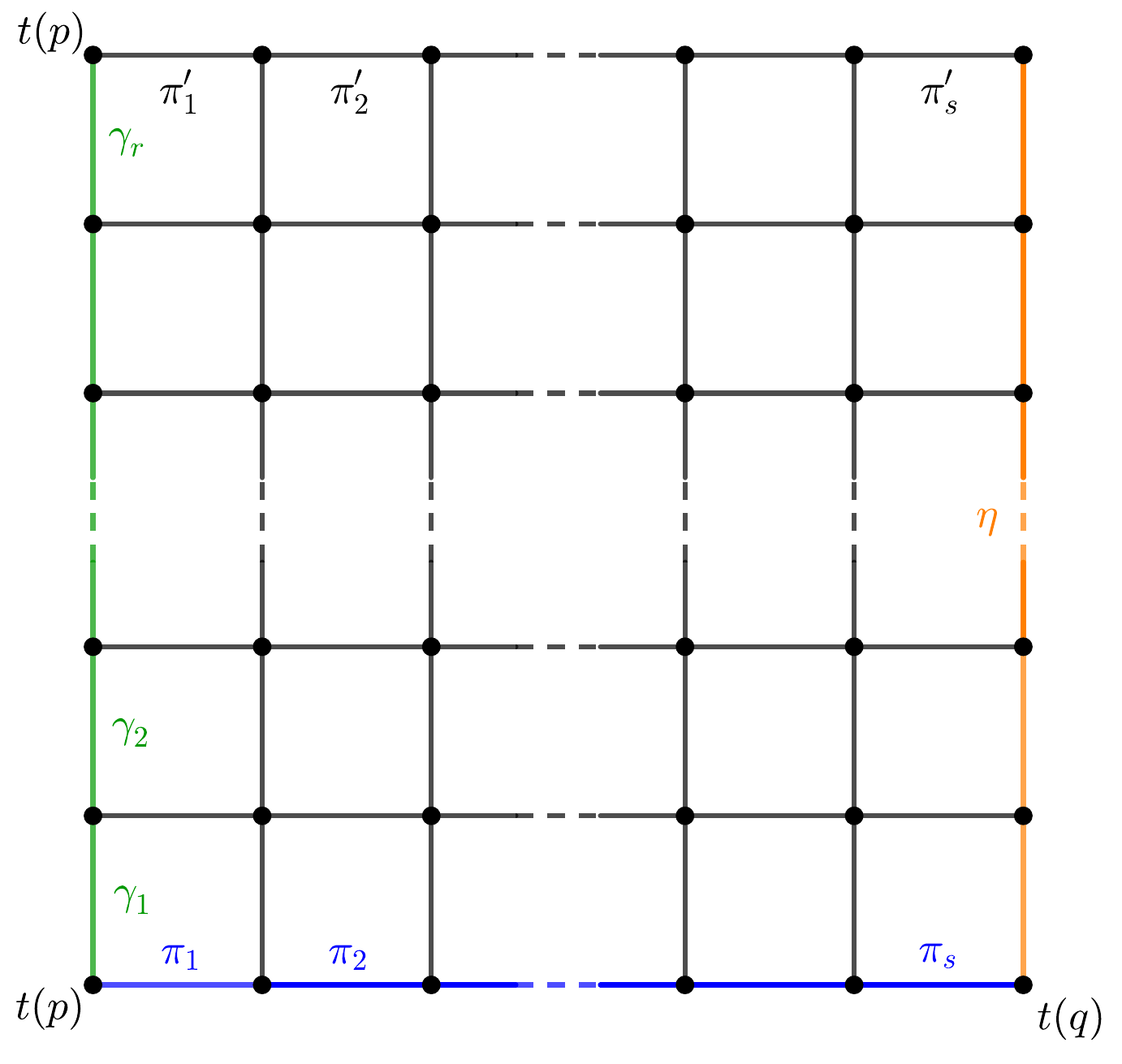}
\end{center}
Because $\gamma_1$ and $\pi_1$ belong to transverse hyperplanes, they must span a square in $Q$, since otherwise the corresponding hyperplanes of $Q$ would inter-osculate. The same argument shows that the edge of this square opposite to $\gamma_1$ must span another square with $\pi_2$. After finitely many iterations, we conclude that there is a $(1 \times s)$-grid spanned by $\pi$ and $\gamma_1$. Similarly, the path opposite to $\pi$ in this grid must span a $(1 \times s)$-grid with $\gamma_2$. After finitely many iterations, we conclude that $\pi$ and $\gamma$ span an $(r \times s)$-grid. Let $\eta$ denote the path opposite to $\gamma$ in this grid. Clearly, $\eta$ starts from $t(q)$ and crosses the same hyperplanes as $\gamma$ and in the same order. 

\medskip \noindent
In order to conclude that $w$ is $(t(q),t(q))$-legal, it remains to justify that $\eta$ is a loop. Let $\pi_1', \ldots, \pi_s'$ denote the edges opposite to $\pi_1, \ldots, \pi_s$ respectively in our $(r \times s)$-grid. Because $\pi_1$ and $\pi_1'$ belong to the same hyperplane and share an endpoint, they must coincide, since otherwise $Q$ would contain a self-intersecting or self-osculating hyperplane. The same argument shows that $\pi_2$ and $\pi_2'$ coincide. After finitely many iterations, we conclude that $\pi_i$ coincides with $\pi_i'$ for every $1 \leq i \leq s$. This implies that the terminus of $\eta$ is $t(q)$, as desired.
\end{proof}

\subsection{Proof of Theorem~\ref{thm:IntroSpecial}}\label{section:ProofSpecial}

\noindent
This section is dedicated to the proof of Theorem~\ref{thm:IntroSpecial} (i.e.\ Theorem~\ref{thm:SpecialLamp} above). Our goal, given a group $G$ acting on a median graph $X$ in some specific way, will be to produce a hyperbolic cone-off $\dot{X}$ of $X$ such that the canonical map $X \to \dot{X}$ is $x^y$-gentle. Proposition~\ref{prop:QuasiSyl} above shows how quasi-syllabic collections may help. In our median setting, our collections will satisfy an even stronger property:

\begin{definition}
Let $X$ be a graph. A collection of subgraphs $\mathcal{P}$ is \emph{syllabic} if any two vertices of $X$ are connected by some geodesic in $\mathrm{ConeOff}(X, \mathcal{P})$ that extends to a geodesic in $X$.  
\end{definition}

\noindent
Our main criterion in order to produce syllabic collections in median graphs is the following:

\begin{prop}\label{prop:SyllabicMedian}
Let $X$ be a median graph and $\mathcal{P}$ a collection of convex subgraphs. Assume that, for any two parallel pairs of vertices $(x,y)$ and $(a,b)$, if $x$ and $y$ both belong to some subgraph in $\mathcal{P}$, then $a$ and $y$ also belong to some subgraph in $\mathcal{P}$. Then, the collection $\mathcal{P}$ is syllabic.
\end{prop}

\noindent
We postpone the proof of Proposition~\ref{prop:SyllabicMedian} to Section~\ref{section:SyllabicQM}, where we will state and prove a similar statement in the more general setting of quasi-median graphs (see Proposition~\ref{prop:SyllabicQM}). 

\begin{proof}[Proof of Theorem~\ref{thm:SpecialLamp}.]
The implication $(i) \Rightarrow (ii)$ follows from the fact that every $\mathbb{F}_2$-subgroup in a right-angled Artin group is undistorted \cite{MR3358258}, which implies that every $(\mathbb{F}_2 \times \mathbb{F}_2)$-subgroup in a right-angled Artin group, and a fortiori in a cocompact special group, is undistorted. The implication $(ii) \Rightarrow (iii)$ is clear, and $(iii) \Rightarrow (iv)$ follows from the fact that $\mathrm{Lamp}(\mathbb{Z})$ quasi-isometrically embeds into $\mathbb{F}_2  \times \mathbb{F}_2$ (see e.g.\ \cite[Corollary~10]{MR3079268}).  The implication $(iv) \Rightarrow (v)$ follows from Theorem~\ref{thm:GentleLamp}, and $(v) \Rightarrow (vi)$ is clear. It remains to show that $(vi) \Rightarrow (i)$, i.e.\ if our group $G$ does not contain $\mathbb{F}_2 \times \mathbb{F}_2$ as a subgroup, then it is $\mathrm{lin}$-polynomially hyperbolic. 

\medskip \noindent
Following the formalism introduced in Section~\ref{section:IntroSpecial}, we think of $G$ as the group $\mathcal{D}(X,o)$, for some compact special cube complex $X$ and some vertex $o \in X^{(0)}$, acting on the median graph $M(X,o)$. Given an $(o,\ast)$-legal word $w$ and a set $\Lambda$ of hyperplanes of $X$, we denote by $S(w,\Lambda) \subset M(X,o)$ the (convex) subgraph induced by the vertices of the form $w u$ where $u$ is a word of hyperplanes from $\Lambda$; and by $P(w,\Lambda) \leq \mathcal{D}(X,o)$ the subgroup given by the $(o,o)$-legal words of the form $wuw^{-1}$ where $u$ is a word of hyperplanes from $\Lambda$. Consider the following collection of convex subgraphs of $M(X,o)$:
$$\mathcal{P}:= \left\{  S(w,\Lambda) \text{ such that } P(w,\Lambda) \text{ is abelian} \right\}.$$
The content of the proof of \cite[Theorem~1.1]{SpecialRH} is precisely that $\mathrm{ConeOff}(M(X,o), \mathcal{P})$ is hyperbolic provided that $\mathcal{D}(X,o)$ does not contain $\mathbb{F}_2 \times \mathbb{F}_2$ as a subgroup. Therefore, in order to conclude that $G$ is $\mathrm{lin}$-polynomially hyperbolic, it suffices to show that the canonical map $M(X,o) \to \mathrm{ConeOff}(M(X,o),\mathcal{P})$ is $x^y$-gentle. According to Proposition~\ref{prop:QuasiSyl}, this assertion follows from our next two claims.

\begin{claim}
The collection $\mathcal{P}$ is syllabic.
\end{claim}

\noindent
Our goal is to apply Proposition~\ref{prop:SyllabicMedian}. So let $(x,y)$ and $(a,b)$ be two parallel pairs of vertices in $M(X,o)$. Assume that $x,y \in S(z,\Lambda)$ for some $(o,\ast)$-legal word $z$ and some collection $\Lambda$ of hyperplanes of $X$ such that $P(z,\Lambda)$ is abelian. Notice that $S(z,\Lambda)= S(x,\Lambda)$. Indeed, since $x \in S(z,\Lambda)$, we can write $x=zu$ for some $(t(z),t(x))$-legal word $u$ of hyperplanes in $\Lambda$. Then,
$$\begin{array}{lcl} S(z, \Lambda) & = & z \{ (t(z),\ast)\text{-legal words of hyperplanes in } \Lambda \} \\ \\ & = & zu \cdot u^{-1} \{ (t(z),\ast)\text{-legal words of hyperplanes in } \Lambda \} \\ \\ & = & x  \{ (t(x),\ast)\text{-legal words of hyperplanes in } \Lambda \} = S(x,\Lambda). \end{array}$$
For the same reason, $P(z,\Lambda)= P(x,\Lambda)$. Thus, we can assume that $z=x$. As a consequence of Lemmas~\ref{lem:ParallelPairsHyp} and~\ref{lem:GeodesicsSpecial}, $x^{-1}y$ and $a^{-1}b$ have the same support, say $\Lambda$. Clearly, $a,b \in S(a,\Xi)$. Moreover, since $\mathrm{supp}(x^{-1}a)$ and $\mathrm{supp}(x^{-1}y)$ are transverse as a consequence of Lemma~\ref{lem:ParallelPairsTrans} and Corollary~\ref{cor:LabelTransHyp}, it follows from Lemma~\ref{lem:TransverseLegal} that the conjugation by $xa^{-1}$ in $\mathcal{D}(X)$ induces an isomorphism $P(a,\Xi) \to P(x,\Xi)$. Thus, $P(a,\Xi)$ is isomorphic to the subgroup $P(x,\Xi)$ of the abelian subgroup $P(x,\Lambda)$. We conclude that $P(a,\Xi)$ is abelian, hence $S(a,\Xi) \in \mathcal{P}$, concluding the proof of our claim. 

\begin{claim}
Every subgraph in $\mathcal{P}$ has polynomial growth.
\end{claim}

\noindent
Let $w$ be an $(o,\ast)$-legal word and $\Lambda$ a collection of hyperplanes of $X$. In order to show that $S(w,\Lambda)$ has polynomial growth, it suffices to verify that the abelian group $P(w,\Lambda)$ acts cocompactly on $S(w,\Lambda)$. Let $wu_1, \ldots, wu_n$ be vertices of $S(w,\Lambda)$, where $u_1, \ldots, u_n$ are $(t(w),\ast)$-legal words of hyperplanes in $\Lambda$. If $n$ is larger than the number of vertices in $X$, then there must exist $i \neq j$ such that $t(wu_i)=t(wu_j)$. Then, $wu_iu_j^{-1}w^{-1}$ defines an $(o,o)$-legal word representing an element of $P(w,\Lambda)$ that sends $wu_j$ to $wu_i$. This concludes our proof. 
\end{proof}

\section{Application to graph products}

\noindent
Recall that, given a graph $\Gamma$ and a collection of groups $\mathcal{G}= \{G_u \mid u \in V(\Gamma)\}$ indexed by $V(\Gamma)$, the graph product $\Gamma \mathcal{G}$ is given by the relative presentation
$$\langle G_u \ (u \in V(\Gamma)) \mid [G_u,G_v]=1 \text{ for all edges } \{u,v\} \in E(\Gamma) \rangle,$$
where $[G_u,G_v]=1$ is a shorthand for: $[g,h]=1$ for all $g \in G_u$ and $h \in G_v$. Usually, one says that graph products interpolate between direct sums (when $\Gamma$ is a complete graphs, so ``everything commutes'') and free products (when $\Gamma$ has no edge, so ``nothing commutes''). Examples include right-angled Artin (resp.\ Coxeter) groups, which coincide with graph products of infinite cyclic groups (resp.\ cyclic groups of order two). 

\medskip \noindent
\textbf{Convention.} \emph{From now on, vertex-groups of graph products are always assumed to be non-trivial.}

\medskip \noindent
Our goal is to prove Theorem~\ref{thm:IntroGP} from the introduction, which we restate below for the reader's convenience:

\begin{thm}\label{thm:InGP}
Let $\Gamma$ be a finite graph and $\mathcal{G}$ a collection of groups of polynomial growth. The following statements are equivalent:
\begin{itemize}
	\item[(i)] $\Gamma \mathcal{G}$ contains $\mathbb{F}_2 \times \mathbb{F}_2$ as a subgroup;
	\item[(ii)] $\Gamma$ contains one of the graphs from Figure~\ref{Graphs} as an induced subgraph;
	\item[(iii)] $\Gamma \mathcal{G}$ contains $\mathbb{F}_2 \times \mathbb{F}_2$ as an undistorted subgroup;
	\item[(iv)] $\mathbb{F}_2 \times \mathbb{F}_2$ quasi-isometrically embeds into $\Gamma \mathcal{G}$;
	\item[(v)] $\mathrm{Lamp}(\mathbb{Z})$ quasi-isometrically embeds into $\Gamma \mathcal{G}$;
	\item[(vi)] $\Gamma \mathcal{G}$ is not $\mathrm{pol}$-polynomially hyperbolic;
	\item[(vii)] $\Gamma \mathcal{G}$ is not $\mathrm{lin}$-polynomially hyperbolic.
\end{itemize}
\end{thm}
\begin{figure}[h!]
\begin{center}
\includegraphics[width=0.7\linewidth]{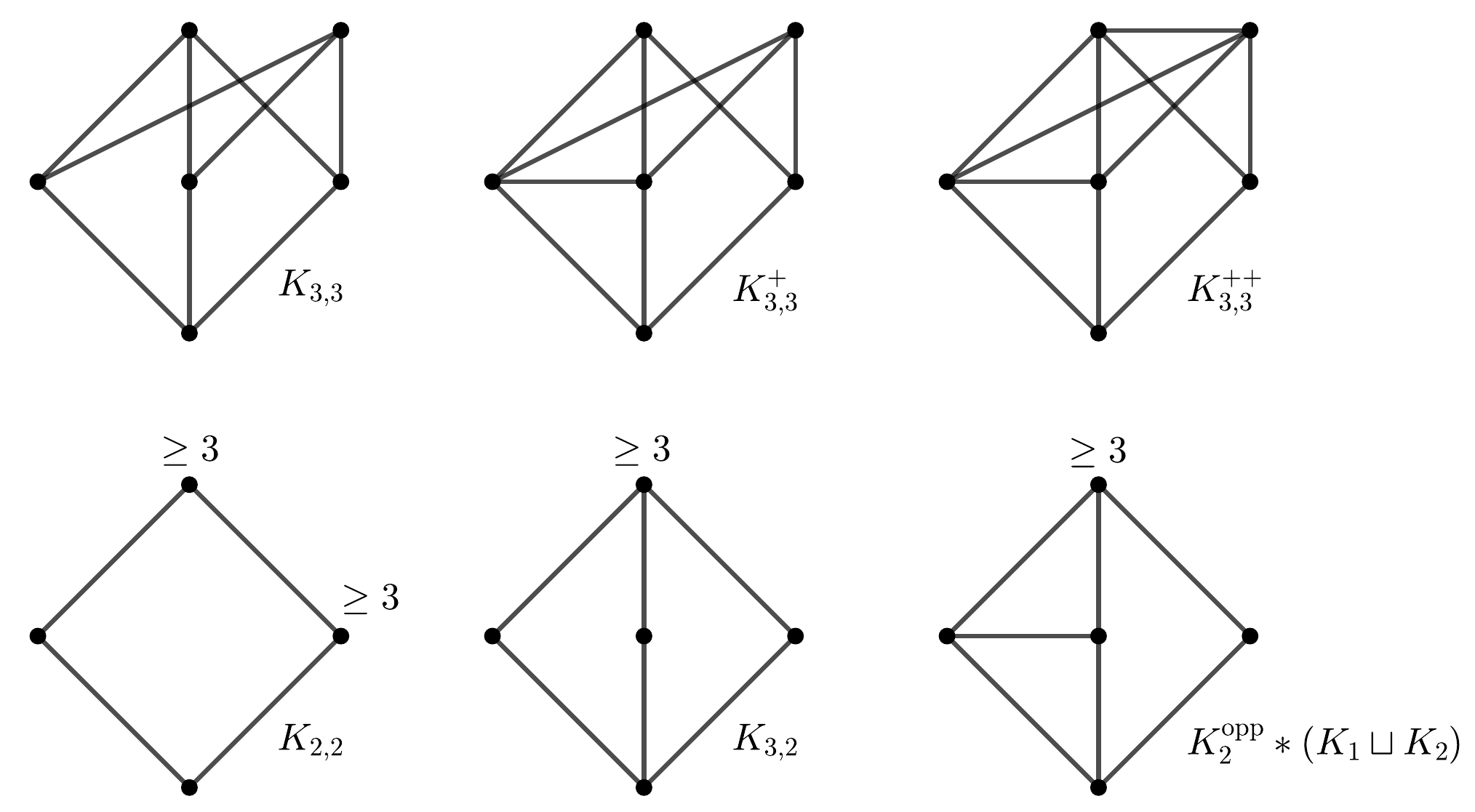}
\caption{}
\label{Graphs}
\end{center}
\end{figure}

\noindent
In Section~\ref{section:QM}, we describe the quasi-median geometry of graph products, as introduced in \cite{QM}. In Sections~\ref{section:HypCrit} and~\ref{section:SyllabicQM}, we state and prove sufficient conditions that will allow us to show that some cone-offs of quasi-median graphs are hyperbolic and syllabic. Finally, we prove Theorem~\ref{thm:InGP} in Section~\ref{section:ThmGP}.

\subsection{Preliminaries}\label{section:QM}

\noindent
In this section, we record some basic definitions and properties related to quasi-median graphs, and we describe their connection with graph products as highlighted in \cite{QM}. 

\medskip \noindent
Quasi-median graphs admit several different but equivalent definitions, see \cite{MR1297190}. The definition we give here echoes the characterisation of median graphs as retracts of hypercubes. (Recall that, given a graph $X$, a subgraph $Y \subset X$ is a \emph{retract} if there exists a graph morphism $X \to Y$ that restricts to the identity on $Y$.) Roughly speaking, we can think of quasi-median graphs as defining the smallest reasonable class of graphs that encompasses median graphs and products of complete graphs.

\begin{definition}
A connected graph is \emph{quasi-median} if it is retract of a Hamming graph (i.e.\ a product of complete graphs). 
\end{definition}

\noindent
This definition is not the most useful in practice, but it highlights the tight connection between median and quasi-median graphs. Similarly to median graphs, the key objects to understand in order to describe the geometry of quasi-median graphs are \emph{hyperplanes}. 

\begin{definition}
In a quasi-median graph, a \emph{hyperplane} is an equivalence class of edges with respect to the reflexive-transitive closure of the relation that identifies two edges whenever they belong to a common $3$-cycle or they are opposite in some induced $4$-cycle. Two hyperplanes are \emph{transverse} if they contain two edges with a common endpoint that span an induced $4$-cycle.
\end{definition}
\begin{figure}
\begin{center}
\includegraphics[trim=1cm 15.5cm 10cm 0,clip,width=0.6\linewidth]{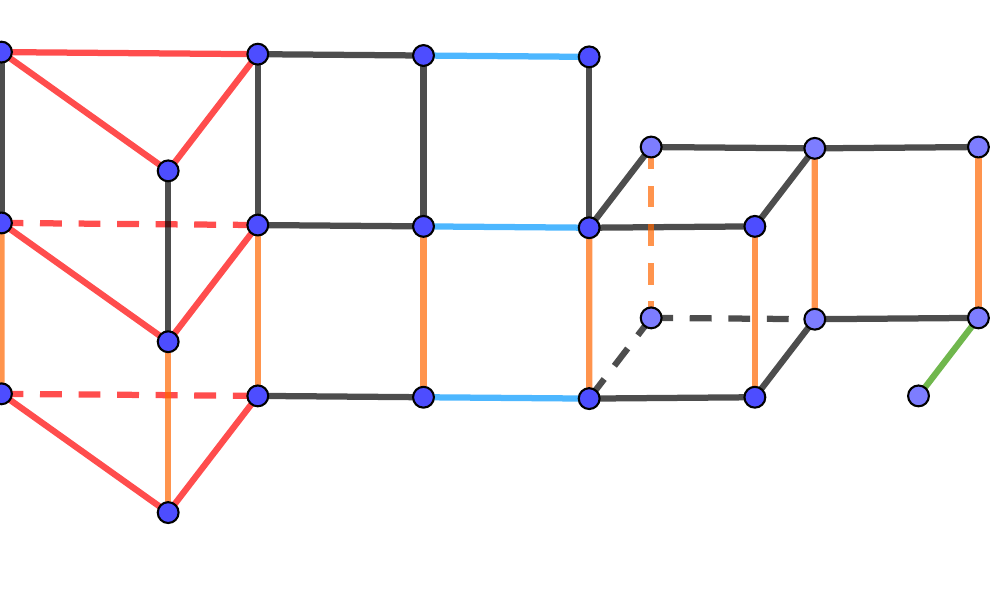}
\caption{Examples of hyperplanes in a quasi-median graph. The orange hyperplane is transverse to the blue and red hyperplanes.}
\label{HypQM}
\end{center}
\end{figure}

\noindent
See Figure~\ref{HypQM}. Roughly speaking, the geometry of a quasi-median graph is encoded by the combinatorics of its hyperplanes. Our next statement justifies this leitmotiv.

\begin{thm}[{\cite[Proposition~2.15]{QM}}]\label{thm:BigQM}
Let $X$ be a quasi-median graph.
\begin{itemize}
	\item For every hyperplane $J$, the graph $X \backslash \backslash J$ obtained from $X$ by removing the edges of $J$ is disconnected. Its connected components are referred to as \emph{sectors}. 
	\item For every hyperplane $J$, the sectors delimited by $J$, its \emph{carrier} $N(J)$ (i.e.\ the subgraph induced by $J$ in $X$), and its \emph{fibres} (i.e.\ the connected components of $N(J) \backslash \backslash J$) are gated.
	\item A path in $X$ is a geodesic if and only if it crosses each hyperplane at most once.
	\item The distance between two vertices coincides with the number of hyperplanes that separate them. 
\end{itemize}
\end{thm}

\noindent
Recall that, given a graph $X$, a subgraph $Y \subset X$ is \emph{gated} if, for every $x \in V(X)$, there exists a unique vertex $y \in V(Y)$ such that $x$ can be connected to any vertex in $Y$ through some geodesic passing through $y$. The vertex $y$ is referred to as the \emph{gate} or the \emph{projection} of $x$. Notice that, if the gate of $x$ exists, then it coincides with the unique vertex of $Y$ that minimises the distance to $x$. Gatedness can be thought of as a strong convexity condition. It is worth mentioning that the gatedness of sectors in quasi-median graphs given by Theorem~\ref{thm:BigQM} does not strengthen the convexity of halfspaces in median graphs given by Theorem~\ref{thm:BigMedian}. Indeed, in median graphs, a subgraph is gated if and only if it is convex.

\medskip \noindent
In the same way that median graphs can be thought of as made of cubes, quasi-median graphs can be thought of as made of \emph{prisms}.

\begin{definition}
In a quasi-median graph $X$, a clique is a maximal complete subgraph and a prism is a product of cliques. The cubical dimension of $X$, denoted by $\mathrm{dim}_\square(X)$, is the maximal number of factors in a prism of $X$.  
\end{definition}

\noindent
We emphasize that the cubical dimension may differ from the dimension of the prism complex associated to a quasi-median graph. For instance, the cubical dimension of the quasi-median graph illustrated by Figure~\ref{HypQM} is $3$, but the cubical dimension of a complete graph is $1$ regardless of its size. 

\medskip \noindent
Clearly, three vertices in a quasi-median graph may not admit a median vertex. This already happens in complete graphs. This motivates the following definition:

\begin{definition}
Let $X$ be a graph and $(x_1,x_2,x_3) \in V(X)^3$ three vertices. A triple $(y_1,y_2,y_3) \in V(X)^3$ is a \emph{median triangle} of $(x_1,x_2,x_3)$ if 
$$d(x_i,x_j)=d(x_i,y_i)+d(y_i,y_j)+d(y_j,x_j) \text{ for all } i \neq j.$$
\end{definition}

\noindent
Clearly, $(x_1,x_2,x_3)$ is always a median triangle of $(x_1,x_2,x_3)$, so median triangles always exist. Typically, we are interested in finding median triangles as small as possible. Median graphs are characterised by the property that median triangles are unique and reduced to single vertices. In quasi-median graphs, median triangles are also unique and ``not too big''.

\begin{prop}[{\cite[Proposition~2.84]{QM}}]
Let $X$ be a quasi-median graph. Every triple of vertices $(x_1,x_2,x_3)$ admits a unique median triangle $(y_1,y_2,y_3)$. Moreover, the gated hull of $\{y_1,y_2,y_3\}$ is a prism. 
\end{prop}

\noindent
The connection between quasi-median graphs and graph products is highlighted by the following statement:

\begin{thm}[{\cite[Proposition~8.2 and Corollary~8.7]{QM}}]
Let $\Gamma$ be a graph and $\mathcal{G}$ a collection of groups indexed by $V(\Gamma)$. The Cayley graph
$$\mathrm{QM}(\Gamma, \mathcal{G}):= \mathrm{Cayl} \left( \Gamma \mathcal{G}, \bigcup\limits_{G \in \mathcal{G}} G \right)$$
is quasi-median. Its cubical dimension equals the maximal size $\mathrm{clique}(\Gamma)$ of a clique in~$\Gamma$.
\end{thm}

\noindent
Conversely, it turns out that every quasi-median graph can be realised as a gated subgraph in the quasi-median graph of some graph product. See \cite{MR4586831} for more details. 

\medskip \noindent
Recall that, given a graph $\Gamma$ and collection of groups $\mathcal{G}$ indexed by $V(\Gamma)$, a \emph{word} in $\Gamma \mathcal{G}$ is a product $g_1 \cdots g_n$ where $n \geq 0$ and where, for every $1 \leq i \leq n$, $g_i \in G$ for some $G \in \mathcal{G}$; the $g_i$'s are the \emph{syllables} of the word, and $n$ is the \emph{length} of the word. The following elementary operations on a word do not modify the element of $\Gamma \mathcal{G}$ it represents:
\begin{description}
	\item[Cancellation:] delete the syllable $g_i$ if $g_i=1$;
	\item[Merging:] if $g_i,g_{i+1} \in G$ for some $G \in \mathcal{G}$, replace the two syllables $g_i$ and $g_{i+1}$ by the single syllable $g_ig_{i+1} \in G$;
	\item[Shuffling:] if $g_i$ and $g_{i+1}$ belong to two adjacent vertex-groups, shuffle them.
\end{description}
A word is \emph{graphically reduced} if its length cannot be shortened by applying these elementary moves. Every element of $\Gamma \mathcal{G}$ can be represented by a graphically reduced word, and this word is unique up to the shuffling operation. For more information on graphically reduced words, we refer to \cite{GreenGP} (see also \cite{HsuWise,VanKampenGP}). 

\medskip \noindent
Since graphically reduced words coincide with words of minimal length, essentially by definition of Cayley graphs we get:

\begin{lemma}[{\cite[Lemma~8.3]{QM}}]\label{lem:GeodInQM}
Let $\Gamma$ be a graph, $\mathcal{G}$ a collection of groups indexed by $V(\Gamma)$, and $x,y \in \Gamma \mathcal{G}$. If $u_1 \cdots u_n$ is a graphically reduced word representing $x^{-1}y$, then
$$x, \ xu_1, \ xu_1u_2, \ \ldots, \ xu_1u_2 \cdots u_{n-1}, \ xu_1u_2 \cdots u_{n-1}u_n=y$$
is a geodesic in $\mathrm{QM}(\Gamma, \mathcal{G})$. Conversely, every geodesic in $\mathrm{QM}(\Gamma, \mathcal{G})$ connecting $x$ to $y$ is of this form. 
\end{lemma}

\noindent
As a Cayley graph, the edges of $\mathrm{QM}(\Gamma, \mathcal{G})$ are naturally labelled by generators in $\bigcup_{u \in V(\Gamma)} G_u$, and consequently by vertices of $V(\Gamma)$. This $V(\Gamma)$-colouring of the edges of $\mathrm{QM}(\Gamma, \mathcal{G})$ is a useful tool. A basic but fundamental property it satisfies is:

\begin{lemma}[{\cite[Lemma~8.8]{QM}}]\label{lem:LabelHyp}
Let $\Gamma$ be a graph and $\mathcal{G}$ a collection of groups indexed by $V(\Gamma)$. Two edges that belong to the same hyperplane in $\mathrm{QM}(\Gamma, \mathcal{G})$ are indexed by the same vertex of $\Gamma$. 
\end{lemma}

\noindent
As a consequence, the $V(\Gamma)$-colouring of the edges of $\mathrm{QM}(\Gamma, \mathcal{G})$ induces a $V(\Gamma)$-colouring of the hyperplanes of $\mathrm{QM}(\Gamma, \mathcal{G})$ by $V(\Gamma)$. The following observation will be useful later:

\begin{lemma}[{\cite[Lemma~8.12]{QM}}]\label{lem:LabelTransHyp}
Let $\Gamma$ be a graph and $\mathcal{G}$ a collection of groups indexed by $V(\Gamma)$. Two transverse hyperplanes of $\mathrm{QM}(\Gamma, \mathcal{G})$ are labelled by adjacent vertices of $\Gamma$. 
\end{lemma}

\noindent
Let $\Gamma$ be a graph and $\mathcal{G}$ a collection of groups indexed by $V(\Gamma)$. Given a subgraph $\Lambda \subset \Gamma$, we denote by $\langle \Lambda \rangle$ the subgroup $\langle G_u \mid u \in V(\Lambda) \rangle$. Such a subgroup is a \emph{standard parabolic subgroup} of $\Gamma \mathcal{G}$. \emph{Parabolic subgroups} are conjugates of standard parabolic subgroups. 

\begin{lemma}[{\cite[Corollary~6.6]{Mediangle}}]\label{lem:ParabolicGated}
Let $\Gamma$ be a graph and $\mathcal{G}$ a collection of groups indexed by $V(\Gamma)$. For every $\Lambda \subset \Gamma$, the subgraph of $\mathrm{QM}(\Gamma, \mathcal{G})$ induced by $\langle \Lambda \rangle$ is gated. 
\end{lemma}

\subsection{A hyperbolicity criterion}\label{section:HypCrit}

\noindent
This section is dedicated to the proof of the following criterion, which will allow us to show that some cone-offs of quasi-median graphs are hyperbolic. 

\begin{thm}\label{thm:WhenHyp}
Let $X$ be a quasi-median graph of finite cubical dimension and $\mathcal{P}$ a collection of isometrically embedded vertex-sets. Set $Y:= \mathrm{ConeOff}(X,\mathcal{P})$. Assume that there exists a constant $K \geq 0$ such that, for every flat rectangle $[0,a] \times [0,b] \hookrightarrow X$, either the Hausdorff distance in $Y$ between $[0,a] \times \{i\}$ and $[0,a] \times \{j\}$ is $\leq K$ for all $0 \leq i,j \leq b$, or the Hausdorff distance in $Y$ between $\{i\} \times [0,b]$ and $\{j\} \times [0,b]$ is $\leq K$ for all $0 \leq i,j \leq a$. Then $Y$ is hyperbolic.
\end{thm}

\noindent
Our proof follows the lines of \cite[Theorem~4.1]{coningoff}, where a similar statement is proved for median graphs, and uses the following criterion:

\begin{prop}\label{prop:WhenHyp}\emph{(\cite[Proposition~3.1]{Bowditchcriterion})}
For every $D \geq 0$, there exists some $\delta \geq 0$ such that the following holds. Let $T$ be a graph. Assume that a connected subgraph $\eta(x,y)$, containing $x$ and $y$, is associated to each pair of vertices $(x,y) \in T^2$ such that:
\begin{itemize}
	\item for all vertices $x,y \in T$, $d(x,y) \leq 1$ implies $\mathrm{diam}(\eta(x,y)) \leq D$;
	\item for all vertices $x,y,z \in T$, the inclusion $\eta(x,y) \subset \left( \eta(x,z) \cup \eta(z,y) \right)^{+D}$ holds.
\end{itemize}
Then $T$ is $\delta$-hyperbolic.
\end{prop}

\noindent
 Here, given a metric space $X$, a subset $Y \subset X$, and a constant $D \geq 0$, we denote by $Y^{+D}$ the $D$-neighbourhood of $Y$, i.e.\ $Y^{+D}:= \{ x \in X \mid d(x,Y) \leq D \}$.  

\medskip \noindent
In order to state our next preliminary lemma, we need the following definition:

\begin{definition}
A \emph{staircase} is a graph isomorphic to a subgraph in the grid $\mathbb{E}^2$ delimited by $[0,a] \times \{0\}$, $\{0\} \times [0,b]$, and some geodesic $\alpha$ between $(a,0)$ and $(0,b)$ where $a,b \geq 0$. Its \emph{corner} refers to the vertex corresponding to $(0,0)$, and its \emph{broken path} to the geodesic corresponding to $\alpha$.
\end{definition}

\noindent
The next observation will be fundamental in the proof of Theorem~\ref{thm:WhenHyp} below, and is of independent interest.

\begin{lemma}\label{lem:Interval}
Let $X$ be a quasi-median graph, $[x,y]$ a geodesic between two vertices $x,y \in V(X)$, and $z \in I(x,y)$ a third vertex. Then there exists an isometrically embedded staircase $E \subset X$ such that $z$ is its corner and such that its broken path is contained in~$[x,y]$.
\end{lemma}

\noindent
Here, given two vertices $a,b \in X$, $I(a,b)$ refers to the \emph{interval} between $a$ and $b$, i.e.\ the collection of all the vertices that belong to geodesics between $a$ and $b$.

\begin{proof}[Proof of Lemma~\ref{lem:Interval}.]
Since intervals in quasi-median graphs are median (see e.g.\ \cite[Proposition~2.93]{QM}), our lemma is direct consequence of its median version \cite[Lemma~2.7]{MR4922688}. 
\end{proof}

\begin{proof}[Proof of Theorem~\ref{thm:WhenHyp}.]
The strategy is to define $\eta(x,y)$ as the subgraph in $Y$ induced by the interval $I(x,y)$ in $X$ for all $x,y \in V(X)$ and then to apply Proposition~\ref{prop:WhenHyp}. In order to verify the assumptions, the following observation will be needed:

\begin{claim}\label{claim:ForHyp}
Let $x,y \in V(X)$ be two vertices, $[x,y]$ a geodesic in $X$, and $z \in I(x,y)$ a third vertex. Then $z$ lies in the $2K$-neighbourhood of $[x,y]$ in $Y$.
\end{claim}

\noindent
According to Lemma \ref{lem:Interval}, there exists an isometrically embedded staircase $E \subset X$ such that $z$ is its corner and such that its broken path $[x',y']$ lies in $[x,y]$. Write $E$ as a union of flat rectangles $C_1, \ldots, C_r$ such that, for every $1 \leq i \leq r$, $z$ is a corner of $C_i$ and its opposite vertex lies in $[x,y]$. We index our rectangles so that the length of $C_i \cap [z,x']$ decreases. See Figure \ref{ColorStaircase}. We refer to a geodesic in $E$ parallel to $[z,x']$ (resp. to $[z,y']$) as \emph{vertical} (resp. \emph{horizontal}). By extension, we say that a flat rectangle in $E$ is \emph{vertical} (resp. \emph{horizontal}) if the Hausdorff distance in $Y$ between any two of its vertical (resp. horizontal) geodesics is $\leq K$. By assumption, each $C_i$ is vertical or horizontal. Four cases may happen:
\begin{itemize}
	\item If there exists some $1 \leq i \leq r$ such that $C_i$ is both vertical and horizontal, then the distance in $Y$ between $z$ and the vertex of $C_i$ opposite to it (which belongs to $[x,y]$) must be $\leq 2K$. 
	\item If $C_1, \ldots, C_r$ are all vertical, then in particular $C_r$ is vertical, which implies that $d_Y(z,y') \leq K$.
	\item If $C_1, \ldots, C_r$ are all horizontal, then in particular, $C_1$ is horizontal, which implies that $d_Y(z,x') \leq K$.
	\item If there exists some $1 \leq i \leq r-1$ such that $C_i$ is horizontal and $C_{i+1}$ vertical, or vice-versa, then the flat rectangle $C_i \cap C_{i+1}$ is both vertical and horizontal, so the distance in $Y$ between $z$ and the vertex of $C_i \cap C_{i+1}$ opposite to it (which belongs to $[x,y]$) must be $\leq 2K$.
\end{itemize}
This concludes the proof of our claim.
\begin{figure}
\begin{center}
\includegraphics[scale=0.4]{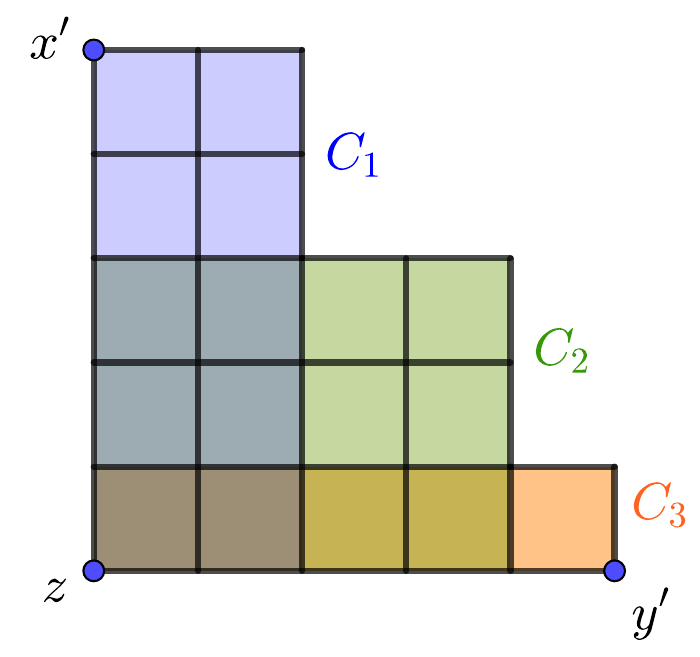}
\caption{Decomposition of the staircase $E$ as a union of flat rectangles.}
\label{ColorStaircase}
\end{center}
\end{figure}

\medskip \noindent
We are now ready to verify that Proposition~\ref{prop:WhenHyp} applies. First, let $x,y \in Y$ be two vertices satisfying $d_Y(x,y) \leq 1$. If $x=y$, then $\eta(x,y)$ is a single vertex; and if $x$ and $y$ are adjacent in $X$, then $\eta(x,y)$ is a single edge. Otherwise, if $d_X(x,y) \geq 2$, then there must exist some $P \in \mathcal{P}$ such that $x,y \in P$. Because $P$ is isometrically embedded, there exists a geodesic $\gamma \subset P$ between $x$ and $y$, and, as a consequence of Claim~\ref{claim:ForHyp}, $d_Y(z, \gamma) \leq 2K$. Therefore, 
$$\mathrm{diam}_Y(\eta(x,y)) \leq 4K+ \mathrm{diam}_Y(\gamma) =4K+1.$$
Thus, we have proved that $\mathrm{diam}_Y(\eta(x,y)) \leq 4K+1$ for all vertices $x,y \in V(Y)$ satisfying the inequality $d_Y(x,y) \leq 1$. 

\medskip \noindent
Next, fix three vertices $x,y,z \in V(Y)$ and a vertex $w \in \eta(x,y)$. Let $(x',y',z')$ denote the median triangle of $(x,y,z)$ and let $[x,x']$, $[x',y']$, $[y',y]$ be geodesics in $X$. Because $[x,x'] \cup [x',y'] \cup [y',y]$ defines a geodesic in $X$, it follows from Claim~\ref{claim:ForHyp} that there exists a vertex $w' \in [x,x'] \cup [x',y'] \cup [y',y]$ such that $d_Y(w,w') \leq 2K$. Because $[x,x'] \subset \eta(x,z)$ and $[y',y] \subset \eta(y,z)$, and because the gated hull of $(x',y',z')$ is a prism of diameter $\leq \mathrm{dim}_\square(X)$ in $X$, it follows that $z$ lies in the $(2K+ \mathrm{dim}_\square(X))$-neighbourhood of $\eta(x,z) \cup \eta(y,z)$ in $Y$, as desired. 

\medskip \noindent
Thus, Proposition~\ref{prop:WhenHyp} applies and we conclude that $Y$ is hyperbolic.
\end{proof}

\subsection{Syllabic collections in quasi-median graphs}\label{section:SyllabicQM}

\noindent
In order to prove Theorem~\ref{thm:InGP}, it will be crucial to find syllabic collections in Cayley graphs of graph products. Proposition~\ref{prop:SyllabicQM} below yields a criterion to verify that some collections are syllabic in quasi-median graphs. But Cayley graphs of graph products with respect to finite generating sets are usually not quasi-median. In order to overcome this difficulty, we introduce the following strengthening of syllabic collections:

\begin{definition}
Let $X$ be a graph. A collection of subgraphs $\mathcal{P}$ is \emph{strongly syllabic} if every geodesic in $\mathrm{ConeOff}(X,\mathcal{P})$ extends to a geodesic in $X$.
\end{definition}

\noindent
Our definition is motivated by the following observation:

\begin{lemma}\label{lem:SyllStrongSyll}
Let $X$ be a graph and $\mathcal{P},\mathcal{Q}$ two collections of subgraphs. If $\mathcal{P}$ is strongly syllabic in $X$ and if $\mathcal{Q}$ is syllabic in $\mathrm{ConeOff}(X,\mathcal{P})$ is syllabic, then $\mathcal{P} \cup \mathcal{Q}$ is syllabic in $X$.
\end{lemma}

\begin{proof}
Let $x,y \in V(X)$ be two vertices. Because $\mathcal{Q}$ is syllabic in $\mathrm{ConeOff}(X,\mathcal{P})$, there exists a geodesic $\gamma$ between $x$ and $y$ in $\mathrm{ConeOff}(X,\mathcal{P} \cup \mathcal{Q})$ that extends to a geodesic $\gamma^+$ in $\mathrm{ConeOff}(X,\mathcal{P})$. Because $\mathcal{P}$ is strongly syllabic in $X$, $\gamma^+$ extends to a geodesic in $X$. Thus, we have found a geodesic between $x$ and $y$ in $\mathrm{ConeOff}(X, \mathcal{P} \cup \mathcal{Q})$ that extends to a geodesic in $X$.
\end{proof}

\noindent
Once combined with our next statement, Lemma~\ref{lem:SyllStrongSyll} will allow us to deduce that some collections in Cayley graphs of graph products are syllabic if we are able to show that they are syllabic in the corresponding quasi-median graphs. 

\begin{lemma}\label{lem:StronglySyllGP}
Let $\Gamma$ be a graph and $\mathcal{G}$ a collection of groups indexed by $V(\Gamma)$. For every $u \in V(\Gamma)$, fix a generating set $S_u \subset G_u$. Set $S:= \bigcup_{u \in V(\Gamma)} S_u$. In $\mathrm{Cayl}(\Gamma \mathcal{G},S)$, the collection $\mathcal{F}:= \{ gG_u \mid g \in \Gamma \mathcal{G}, u \in V(\Gamma) \}$ is strongly syllabic. 
\end{lemma}

\begin{proof}
Since $\mathrm{ConeOff}(\mathrm{Cayl}(\Gamma \mathcal{G},S), \mathcal{F})$ coincides with $\mathrm{QM}(\Gamma, \mathcal{G})$, our lemma follows immediately from \cite[Lemma~3.8 and Claim~8.24]{QM}.
\end{proof}

\noindent
We conclude this section by stating and proving a criterion that allows us to verify that some collections are syllabic in quasi-median graphs. First, we need the following definition:

\begin{definition}
Let $X$ be a quasi-median graph. Two pairs of vertices $(a,b),(x,y) \in V(X)^2$ are \emph{parallel} if every sector contains $a$ but not $b$ if and only if it contains $x$ but not $y$. 
\end{definition}

\noindent
Then, our criterion is:

\begin{prop}\label{prop:SyllabicQM}
Let $X$ be a quasi-median graph and $\mathcal{P}$ a collection of gated subgraphs. Assume that, for any two parallel pairs of vertices $(x,y)$ and $(a,b)$, if $x$ and $y$ both belong to some subgraph in $\mathcal{P}$, then $a$ and $y$ also belong to some subgraph in $\mathcal{P}$. Then, the collection $\mathcal{P}$ is syllabic.
\end{prop}

\noindent
Before turning to the proof of our criterion, we state and prove a couple of preliminary lemmas related to quasi-median geometry.

\begin{lemma}\label{lem:NotGeod}
Let $X$ be a quasi-median graph and $\gamma$ a path. If $\gamma$ is not a geodesic, then there exist two edges $e,f \subset \gamma$ delimiting a subpath $\gamma_0$ such that $e,\gamma_0,f$ span either a ladder (Figure~\ref{LadderTower}(a)) or a tower (Figure~\ref{LadderTower}(b)). 
\end{lemma}
\begin{figure}
\begin{center}
\includegraphics[width=\linewidth]{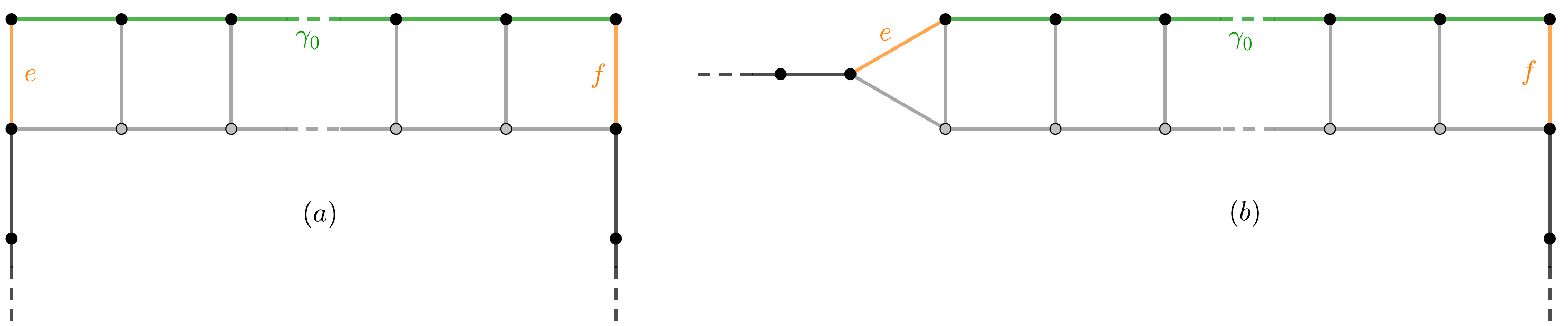}
\caption{A path with a subpath delimiting (a) a ladder and (b) a tower.}
\label{LadderTower}
\end{center}
\end{figure}

\begin{proof}
If $\gamma$ is not a geodesic, then it has to cross some hyperplane twice. Fix two edges $e,f \subset \gamma$ that belong to a common hyperplane $J$ but such that the subpath $\gamma_0 \subset \gamma$ they delimit crosses each hyperplane at most once, i.e.\ $\gamma_0$ is a geodesic. It follows from the convexity of fibres that $\gamma_0$ is contained in a fibre $F$ of $J$. But we know from \cite[Lemma~2.28]{QM} that the carrier of $J$ decomposes as a product $F \times C$, where $C$ denotes the clique of $J$ containing $f$. The desired conclusion then follows easily. 
\end{proof}

\begin{lemma}\label{lem:CliqueInGated}
Let $X$ be a quasi-median graph, $Y \subset X$ a gated subgraph, and $J$ a hyperplane crossing $Y$. For every clique $C \subset J$, if $V(C) \cap V(Y) \neq \emptyset$ then $C \subset Y$.
\end{lemma}

\begin{proof}
Fix a vertex $x \in V(C) \cap V(Y)$ and let $y \in V(C)$ be another vertex. If $y \notin V(Y)$, then $x$ has to be the projection of $y$ on $Y$ since it clearly minimises the distance to $y$ in $Y$. However, we know from \cite[Lemma~2.34]{QM} that a hyperplane separating a vertex from its projection on a given gated subgraph cannot cross this subgraph. Since $J$ crosses $Y$ by assumption, we conclude that $y$ must belong to $Y$. 
\end{proof}

\begin{proof}[Proof of Proposition~\ref{prop:SyllabicQM}.]
Let $x,y \in V(X)$ be two vertices and let $\gamma$ be an arbitrary geodesic in $\mathrm{ConeOff}(X,\mathcal{P})$ connecting $x$ to $y$. We denote by $\gamma^+$ the piecewise geodesic extension of $\gamma$ in $X$, i.e.\ we connect the successive vertices $p_1, \ldots, p_n$ of $\gamma$ with geodesics in $X$. If $\gamma^+$ is a geodesic in $X$, then we have nothing to prove. So assume that $\gamma^+$ is not a geodesic in $X$. According to Lemma~\ref{lem:NotGeod}, there exist two edges $e,f \subset \gamma$ delimiting a subpath $\gamma_0$ such that $e,\gamma_0,f$ span either a ladder or a tower. In the ladder case, we have a configuration of the form:
\begin{center}
\includegraphics[width=0.6\linewidth]{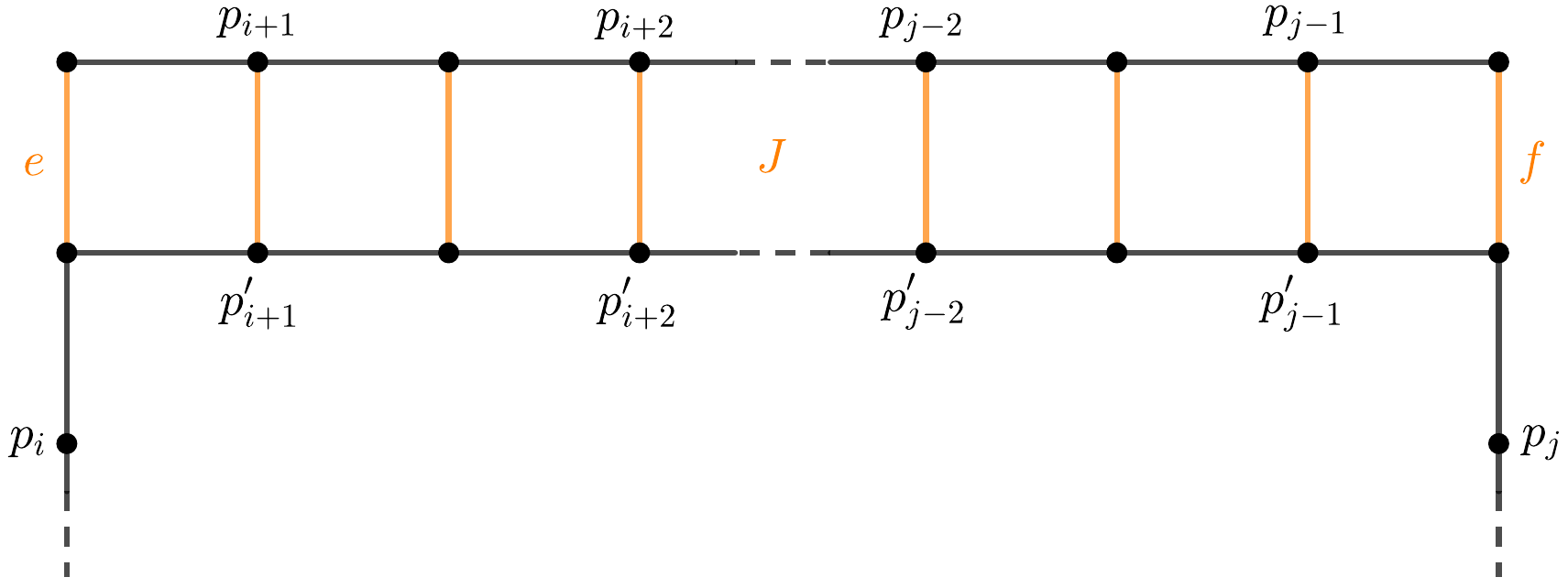}
\end{center}
For every $i+1 \leq s \leq j-1$, let $p_s'$ denote the neighbour of $p_s$ that is separated from $p_s$ by $J$. Notice that, for every $i+1 \leq s \leq j-2$, the pairs $(p_s,p_{s+1})$ and $(p_s',p_{s+1}')$ are parallel, so we know that $p_s'$ and $p_{s+1}'$ are adjacent in $\mathrm{ConeOff}(X,\mathcal{P})$. Notice also that $p_{i+1}'$ lies on a geodesic connecting $p_i$ to $p_{i+1}$, so, because $p_i$ and $p_{i+1}$ both belong to some subgraph in $\mathcal{P}$, which is convex by assumption, we also know that $p_i$ and $p_{i+1}'$ are adjacent in $\mathrm{ConeOff}(X,\mathcal{P})$. Similarly, $p_j$ and $p_{j-1}'$ are adjacent in $\mathrm{ConeOff}(X,\mathcal{P})$. Thus, we can define a new geodesic $\eta$ in $\mathrm{ConeOff}(X,\mathcal{P})$ by replacing $p_{i+1}, \ldots, p_{j-1}$ with $p_{i+1}', \ldots, p_{j-1}'$ in $\gamma$. 

\medskip \noindent
In the tower case, we have a configuration of the form
\begin{center}
\includegraphics[width=0.6\linewidth]{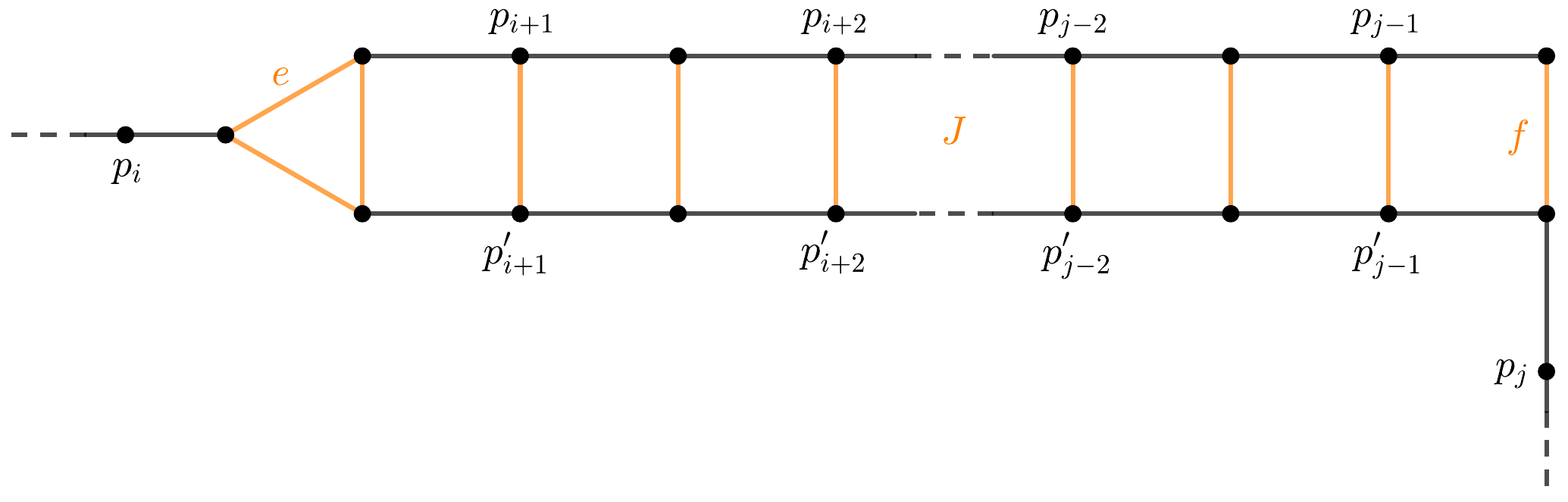}
\end{center}
As before, for every $i+1 \leq s \leq j-1$, we let $p_s'$ denote the neighbour of $p_s$ that is separated from $p_s$ by $J$. And we notice that, for every $i+1 \leq s \leq j-2$, the vertices $p_s'$ and $p_{s+1}'$ are adjacent in $\mathrm{ConeOff}(X,\mathcal{P})$; and that $p_{j-1}'$ and $p_j$ are adjacent in $\mathrm{ConeOff}(X,\mathcal{P})$ as well. The fact that $p_i$ and $p_{i+1}'$ are also connected by an edge in $\mathrm{ConeOff}(X,\mathcal{P})$ follows from Lemma~\ref{lem:CliqueInGated} and from the gatedness of the subgraphs in $\mathcal{P}$. Thus, we can define a new geodesic $\eta$ in $\mathrm{ConeOff}(X,\mathcal{P})$ by replacing $p_{i+1}, \ldots, p_{j-1}$ with $p_{i+1}', \ldots, p_{j-1}'$ in $\gamma$. 

\medskip \noindent
The key observation is that, in both cases, the piecewise geodesic extension $\eta^+$ of $\eta$ in $X$ is shorter than $\gamma^+$. Consequently, after finitely many iterations of our construction, we obtain a geodesic in $\mathrm{ConeOff}(X, \mathcal{P})$ that extends to a geodesic in $X$. 
\end{proof}

\subsection{Proof of Theorem~\ref{thm:IntroGP}}\label{section:ThmGP}

\noindent
This section is dedicated to the proof of Theorem~\ref{thm:IntroGP} from the introduction (i.e.\ Theorem~\ref{thm:InGP} above). We start by proving the implication $(i) \Rightarrow (ii)$, namely:

\begin{prop}\label{prop:ProdInGP}
Let $\Gamma$ be a graph and $\mathcal{G}$ a collection of groups of polynomial growth indexed by $V(\Gamma)$. If $\Gamma \mathcal{G}$ contains $\mathbb{F}_2 \times \mathbb{F}_2$ as a subgroup, then $\Gamma$ contains one of the graphs from Figure~\ref{Graphs} as an induced subgraph. 
\end{prop}

\noindent
In order to prove our proposition, the first ingredient will be the following characterisation of product subgroups in graph products. It follows from \cite{MR3368093} (see \cite[Proposition~2.8]{MR4808711} for a formal statement). 

\begin{lemma}\label{lem:JoinSub}
Let $\Gamma$ be a finite graph and $\mathcal{G}$ a collection of groups indexed by $V(\Gamma)$. If a subgroup $H \leq \Gamma \mathcal{G}$ decomposes as a product of two infinite groups, then $H$ is contained in a conjugate of $\langle \Lambda \rangle$ where $\Lambda$ is either an isolated vertex of $\Gamma$ or a non-trivial join. 
\end{lemma}

\noindent
The following easy observation will be also useful:

\begin{lemma}\label{lem:NormalSubProdFree}
Every non-trivial normal subgroup of $\mathbb{F}_2 \times \mathbb{F}_2$ contains $\mathbb{F}_2$ as a subgroup.
\end{lemma}

\begin{proof}
Let $N \lhd \mathbb{F}_2 \times \mathbb{F}_2$ be a normal subgroup that does not contain $\mathbb{F}_2$ as a subgroup. The intersection between $N$ and $\mathbb{F}_2 \times \{1\}$ yields a normal subgroup of $\mathbb{F}_2$ that does not contain $\mathbb{F}_2$ as a subgroup. Since $\mathbb{F}_2$ does not contain an infinite cyclic normal subgroup, it follows that $N$ intersects trivially $\mathbb{F}_2 \times \{1\}$. Thus, $N$ injects into $\mathbb{F}_2$ as a normal subgroup through the projection $\mathbb{F}_2 \times \mathbb{F}_2 \twoheadrightarrow \mathbb{F}_2$ onto the second factor. Again, due to the fact that $\mathbb{F}_2$ does not contain an infinite cyclic normal subgroup, we conclude that $N$ must be trivial. 
\end{proof}

\noindent
The first step towards Proposition~\ref{prop:ProdInGP} is to understand when our graph products contain non-abelian free groups. This is achieved by our next preliminary lemma.

\begin{lemma}\label{lem:NoFTwoGP}
Let $\Gamma$ be a finite graph and $\mathcal{G}$ a collection of groups indexed by $V(\Gamma)$. If no vertex-group contains $\mathbb{F}_2$ as a subgroup, then $\Gamma \mathcal{G}$ contains $\mathbb{F}_2$ as a subgroup if and only if $\Gamma$ contains one of the graphs from Figure~\ref{NoFTwo} as an induced subgraph. 
\end{lemma}
\begin{figure}[h!]
\begin{center}
\includegraphics[width=0.4\linewidth]{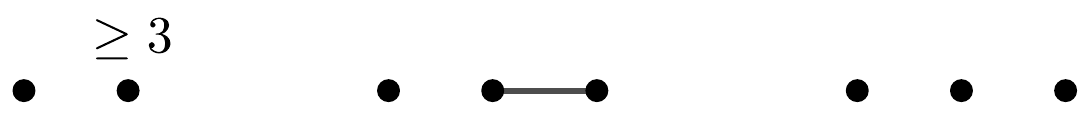}
\caption{}
\label{NoFTwo}
\end{center}
\end{figure}

\begin{proof}
Decompose $\Gamma$ as a join
$$\Gamma= \Gamma_1 \ast \cdots \ast \Gamma_r \ast \Gamma_{r+1} \ast \cdots \ast \Gamma_{r+s}$$
where $\Gamma_1, \ldots, \Gamma_r$ are reduced to single vertices and where $\Gamma_{r+1}, \Gamma_{r+s}$ are not joins but contain at least two vertices. Because $\Gamma \mathcal{G}$ then decomposes as 
$$\Gamma \mathcal{G} = \langle \Gamma_1 \rangle \times \cdots \times \langle \Gamma_r \rangle \times \langle \Gamma_{r+1} \rangle \times \cdots \times \langle \Gamma_{r+s} \rangle,$$
it is clear that $\Gamma \mathcal{G}$ contains $\mathbb{F}_2$ if and only if there exists some $1 \leq i \leq s$ such that $\langle \Gamma_{r+i}$ contains $\mathbb{F}_2$. Thus, we can assume without loss of generality that $\Gamma$ is not a join and contains at least two vertices.

\medskip \noindent
Because $\Gamma$ cannot be a complete graph, it must contain at least two non-adjacent vertices, say $a,b \in V(\Gamma)$. If one of the vertex-groups $G_a$ and $G_b$ has cardinality $\geq 3$, then $\langle G_a,G_b \rangle = G_a \ast G_b$, and a fortiori $\Gamma \mathcal{G}$, contains $\mathbb{F}_2$ as a subgroup. This is one of the configurations given by Figure~\ref{NoFTwo}. Otherwise, if $|G_a|=|G_b|=2$, then two cases may happen: either $V(\Gamma)= \{a,b\}$, in which case $\Gamma \mathcal{G}$ is an infinite dihedral group $\mathbb{D}_\infty = \mathbb{Z}_2 \ast \mathbb{Z}_2$ and does not contain $\mathbb{F}_2$; or there is at least one in $\Gamma$ distinct from both $a$ and $b$. In the latter case, because $\Gamma$ is not a join, we can find a third vertex $c \in V(\Gamma)$ that is not adjacent to both $a$ and $b$. Say that $c$ is not adjacent to $a$. Then, two cases may happen: either $c$ is not adjacent to $b$, in which case $\{a,b,c\}$ provides a configuration from Figure~\ref{NoFTwo} and $\langle G_a,G_b,G_c \rangle = G_a \ast G_b \ast G_c$ contains $\mathbb{F}_2$; or $c$ is adjacent to $b$, in which case $\{a,b,c\}$ also provides a configuration from Figure~\ref{NoFTwo} and $\langle G_a,G_b,G_c \rangle = G_a \ast ( G_b \times G_c)$ also contains $\mathbb{F}_2$. 
\end{proof}

\begin{cor}\label{cor:NoFreeSub}
Let $\Gamma$ be a finite graph and $\mathcal{G}$ a collection of groups indexed by $V(\Gamma)$. If $\Gamma \mathcal{G}$ does not contain $\mathbb{F}_2$ as a subgroup, then $\Gamma$ splits as a join $\Gamma_1 \ast \cdots \ast \Gamma_r \ast \Gamma_{r+1} \ast \cdots \ast \Gamma_{r+s}$ where $\Gamma_1, \ldots, \Gamma_r$ are reduced to single vertices and where $\Gamma_{r+1}, \ldots, \Gamma_{r+s}$ are reduced to pairs of non-adjacent vertices both indexed by $\mathbb{Z}_2$. 
\end{cor}

\begin{proof}
Decompose $\Gamma$ as a join
$$ \Gamma= \Gamma_1 \ast \cdots \ast \Gamma_r \ast \Gamma_{r+1} \ast \cdots \ast \Gamma_{r+s}$$ 
where $\Gamma_1, \ldots, \Gamma_r$ are reduced to single vertices and where $\Gamma_{r+1}, \ldots, \Gamma_{r+s}$ are not joins but contain at least two vertices. Fix some $1 \leq i \leq s$. Because $\Gamma_{r+i}$ cannot be a complete graph, it must contain two non-adjacent vertices, say $a$ and $b$. According to Lemma~\ref{lem:NoFTwoGP}, $G_a$ and $G_b$ must be cyclic of order two. If $\Gamma_{r+i}$ contains a vertex distinct from both $a$ and $b$, then, because $\Gamma_{r+i}$ is not a join, we can find a third vertex $c$ that is not adjacent to both $a$ and $b$. But, then, $\Gamma_{r+i}$ contains one of the graphs from Figure~\ref{NoFTwo} as an induced subgraph, which is impossible according to Lemma~\ref{lem:NoFTwoGP}. Thus, $\Gamma_{r+i}$ is reduced to a pair of non-adjacent vertices both labelled by $\mathbb{Z}_2$. 
\end{proof}

\begin{proof}[Proof of Proposition~\ref{prop:ProdInGP}.]
Because $\mathbb{F}_2 \times \mathbb{F}_2$ is finitely generated, we can assume that $\Gamma$ is finite up to replacing $\Gamma \mathcal{G}$ with one of its parabolic subgroups. Then, we can argue by induction on the number of vertices of $\Gamma$. 

\medskip \noindent
According to Lemma~\ref{lem:JoinSub}, there exists a (non-trivial) join $\Phi \ast \Psi \subset \Gamma$ such that the subgroup $\langle \Phi \ast \Psi \rangle$ contains a conjugate of our $\mathbb{F}_2 \times \mathbb{F}_2$. If $\langle \Phi \rangle$ and $\langle \Psi \rangle$ both contain $\mathbb{F}_2$ as a subgroup, then it follows from Lemma~\ref{lem:NoFTwoGP} that $\Gamma$ contains an induced copy of a join $\Omega \ast \Xi$ where $\Omega,\Xi$ belong to the graphs from Figure~\ref{NoFTwo}. Such a join provides a graph from Figure~\ref{Graphs}, as desired. Otherwise, if one of $\langle \Phi \rangle$ and $\langle \Psi \rangle$ does not contain $\mathbb{F}_2$ as a subgroup, say $\langle \Phi \rangle$, then it follows from Corollary~\ref{cor:NoFreeSub} that $\langle \Phi \rangle$ has polynomial growth and we deduce from Lemma~\ref{lem:NormalSubProdFree} that $\mathbb{F}_2 \times \mathbb{F}_2$ embeds into $\langle \Psi \rangle$ through the quotient map $\langle \Phi \ast \Psi \rangle = \langle \Phi \rangle \times \langle \Psi \rangle \twoheadrightarrow \langle \Psi \rangle$. But $\Psi$ is proper subgraph of $\Gamma$, so we can conclude by induction.
\end{proof}

\noindent
Next, we verify the implication $(ii) \Rightarrow (iii)$ from Theorem~\ref{thm:InGP}. 

\begin{lemma}\label{lem:ProdFreeSub}
Let $\Gamma$ be a finite graph and $\mathcal{G}$ a collection of groups of polynomial growth indexed by $V(\Gamma)$. If $\Gamma$ contains one of the graphs from Figure~\ref{Graphs} as an induced subgraph, then $\Gamma \mathcal{G}$ contains $\mathbb{F}_2 \times \mathbb{F}_2$ as an undistorted subgroup.
\end{lemma}

\begin{proof}
Notice that every graph from Figure~\ref{Graphs} is a join between two graphs from Figure~\ref{NoFTwo}. Therefore, it suffices to verify that, if $\Gamma$ contains one of the graphs from Figure~\ref{NoFTwo} as an induced subgraph, then $\Gamma \mathcal{G}$ contains $\mathbb{F}_2$ as an undistorted subgroup. 

\medskip \noindent
If $\Gamma$ contains one of the graphs from Figure~\ref{NoFTwo} as an induced subgraph, then $\Gamma \mathcal{G}$ contains a parabolic subgroup of the form $A \ast B$ with vertex-groups $A,B$ satisfying $|A|\geq 2, |B| \geq 3$, or $\mathbb{Z}_2 \ast \mathbb{Z}_2 \ast \mathbb{Z}_2$, or $\mathbb{Z}_2 \ast (\mathbb{Z}_2 \times \mathbb{Z}_2)$. Since parabolic subgroups are retracts, they are undistorted, so, in order to conclude the proof of our lemma, it suffices to verify that these three types of groups contain $\mathbb{F}_2$ as an undistorted subgroup. For $\mathbb{Z}_2 \ast \mathbb{Z}_2 \ast \mathbb{Z}_2$ and $\mathbb{Z}_2 \ast (\mathbb{Z}_2 \times \mathbb{Z}_2)$, this is clear since these groups are virtually free of rank $\geq 2$. So consider $A \ast B$. Notice that $A$ and $B$, as vertex-groups, have polynomial growth, and consequently are virtually nilpotent. As such, they are either finite or they contain an undistorted infinite-order element (an infinite finitely generated virtually nilpotent group always virtually surjects onto $\mathbb{Z}$, so an element that maps non-trivially to $\mathbb{Z}$ yields an undistorted infinite-order element). As a consequence, $A \ast B$ contains an undistorted subgroup isomorphic to $F_1 \ast F_2$ where $F_1,F_2$ are two finite groups satisfying $|F_1|\geq 2, |F_2| \geq 3$; or to $\mathbb{Z}\ast F_3$ for some non-trivial finite group $F_3$; or to $\mathbb{Z} \ast \mathbb{Z}$. Since all these groups are virtually free of rank $\geq 2$, the desired conclusion follows. 
\end{proof}

\noindent
We are now ready to prove Theorem~\ref{thm:InGP}. 

\begin{proof}[Proof of Theorem~\ref{thm:InGP}.]
The implications $(i) \Rightarrow (ii) \Rightarrow (iii)$ are given by Proposition~\ref{prop:ProdInGP} and Lemma~\ref{lem:ProdFreeSub}; $(iii) \Rightarrow (iv)$ is clear; $(iv) \Rightarrow (v)$ follows from the fact that $\mathrm{Lamp}(\mathbb{Z})$ quasi-isometrically embeds into $\mathbb{F}_2 \times \mathbb{F}_2$ (see e.g.\ \cite[Corollary~10]{MR3079268}); $(v) \Rightarrow (vi)$ is given by Theorem~\ref{thm:GentleLamp}; and $(vi) \Rightarrow (vii)$ is clear. It remains to prove that $(vii) \Rightarrow (i)$, i.e.\ $\Gamma \mathcal{G}$ is $\mathrm{lin}$-polynomially hyperbolic if it does not contain $\mathbb{F}_2 \times \mathbb{F}_2$ as a subgroup. 

\medskip \noindent
For every $u \in V(\Gamma)$, fix a finite generating set $S_u \subset G_u$. Set $S:= \bigcup_{u \in V(\Gamma)} S_u$ and denote by $\mathrm{Cayl}(\Gamma \mathcal{G})$ the Cayley graph of $\Gamma \mathcal{G}$ with respect to $S$. Consider the collection
$$\mathcal{P}:= \{ \text{cosets of standard parabolic subgroups of polynomial growth} \}.$$
Notice that, because parabolic subgroups are undistorted in $\Gamma \mathcal{G}$ (as retracts), the polynomial growth in the definition of $\mathcal{P}$ may refer equivalently to the growth in $\Gamma \mathcal{G}$ or to the intrinsic growth of the parabolic subgroups (as finitely generated groups).

\medskip \noindent
Our goal is to show that the canonical map $\mathrm{Cayl}(\Gamma \mathcal{G}) \to \mathrm{ConeOff}(\mathrm{Cayl}(\Gamma \mathcal{G}),\mathcal{P})$ yields an $x^y$-gentle map to a hyperbolic graph. 

\medskip \noindent
The hyperbolicity of the cone-off is provided by Claim~\ref{claim:HypConeOff} below. According to Lemma~\ref{lem:StronglySyllGP}, $\mathcal{F}:= \{ gG_u \mid g \in \Gamma \mathcal{G}, u \in V(\Gamma)\}$ is strongly syllabic in $\mathrm{Cayl}(\Gamma \mathcal{G})$, and, according to Claim~\ref{claim:SyllabicGP} below, $\mathcal{P}$ is syllabic in $\mathrm{QM}(\Gamma, \mathcal{G})=\mathrm{ConeOff}(\mathrm{Cayl}(\Gamma \mathcal{G}), \mathcal{F})$. It follows from Lemma~\ref{lem:SyllStrongSyll} that $\mathcal{P}=\mathcal{P} \cup \mathcal{F}$ is syllabic in $\mathrm{Cayl}(\Gamma \mathcal{G})$. Thus, Proposition~\ref{prop:QuasiSyl} implies that the canonical map $\mathrm{Cayl}(\Gamma \mathcal{G}) \to \mathrm{ConeOff}(\mathrm{Cayl}(\Gamma \mathcal{G}),\mathcal{P})$ is $\gamma_\mathcal{P}(x)^y$-gentle. 

\medskip \noindent
Thus, in order to conclude the proof of our theorem, it suffices to justify that $\mathcal{P}$ has polynomial growth in $\Gamma \mathcal{G}$. But we already know, by definition of $\mathcal{P}$, that each member of $\mathcal{P}$ has polynomial growth. Since $\Gamma$ is finite, there are only finitely many cosets of standard parabolic subgroups in $\Gamma \mathcal{G}$ up to $\Gamma \mathcal{G}$-translations. Therefore, the members of $\mathcal{P}$ must have uniform polynomial growth, which amounts to saying that the collection $\mathcal{P}$ has polynomial growth. This concludes the proof of our theorem. 

\begin{claim}\label{claim:HypConeOff}
The graph $\mathrm{ConeOff}(\mathrm{Cayl}(\Gamma \mathcal{G}),\mathcal{P})$ is hyperbolic.
\end{claim}

\noindent
Since $\mathrm{ConeOff}(\mathrm{Cayl}(\Gamma \mathcal{G}), \mathcal{P})$ coincides with $\mathrm{ConeOff}(\mathrm{QM}(\Gamma, \mathcal{G}), \mathcal{P})$, it suffices to show that the cone-off of $\mathrm{QM}(\Gamma , \mathcal{G})$ is hyperbolic thanks to Theorem~\ref{thm:WhenHyp}. First of all, notice that every member of $\mathcal{P}$ is isometrically embedded, since gated according to Lemma~\ref{lem:ParabolicGated}.

\medskip \noindent
Next, let $[0,a] \times [0,b] \hookrightarrow \mathrm{QM}(\Gamma , \mathcal{G})$ be a flat rectangle. Let $\Phi$ (resp.\ $\Psi$) denote the subgraph of $\Gamma$ induced by the vertices associated to the generators labelling the edges of $[0,a] \times \{0\}$ (resp.\ $\{0\} \times [0,b]$). As a consequence of Lemma~\ref{lem:LabelTransHyp}, every vertex of $\Phi$ is adjacent to every vertex of $\Psi$. This implies that the parabolic subgroup $\langle \Phi \cup \Psi \rangle$ splits as $\langle \Phi \rangle \times \langle \Psi \rangle$. Since $\Gamma \mathcal{G}$ does not contain $\mathbb{F}_2 \times \mathbb{F}_2$ as a subgroup, at least one of $\langle \Phi \rangle$ and $\langle \Psi \rangle$ does not contain $\mathbb{F}_2$ as a subgroup. Say that $\langle \Phi \rangle$ does not contain $\mathbb{F}_2$. As a consequence of Corollary~\ref{cor:NoFreeSub}, $\langle \Phi \rangle$ decomposes as a product of vertex-groups and infinite dihedral groups, which implies that $\langle \Phi \rangle$ has polynomial growth. Thus, since we know from Lemma~\ref{lem:LabelHyp} that every segment $[0,a] \times \{i\}$ has all its edges labelled by generators coming from $\Phi$, we deduce that $(i,c)$ and $(j,c)$ are adjacent in $\mathrm{ConeOff}(\mathrm{QM}(\Gamma, \mathcal{G}), \mathcal{P})$ for all $0 \leq i,j \leq a$ and $0 \leq c \leq b$. In particular, the Hausdorff distance between $\{i\} \times [0,b]$ and $\{j\} \times [0,b]$ is $\leq 1$ for all $0 \leq i,j \leq n$. 

\medskip \noindent
Thus, we have proved that Theorem~\ref{thm:WhenHyp} applies, which concludes the proof of our claim. 

\begin{claim}\label{claim:SyllabicGP}
The collection $\mathcal{P}$ is syllabic in $\mathrm{QM}(\Gamma, \mathcal{G})$.
\end{claim}

\noindent
Our goal is to apply Proposition~\ref{prop:SyllabicQM}. First of all, notice that, according to Lemma~\ref{lem:ParabolicGated}, every member of $\mathcal{P}$ is gated. Next, let $(x,y)$ and $(a,b)$ be two parallel pairs of vertices in $\mathrm{QM}(\Gamma, \mathcal{G})$. Assume that $x$ and $y$ both belong to some member of $\mathcal{P}$, say $x,y \in g \langle \Lambda \rangle$ for some $g \in \Gamma \mathcal{G}$ and $\Lambda \subset \Gamma$ with $\langle \Lambda \rangle$ of polynomial growth. Since $x^{-1}y \in \langle \Lambda \rangle$, we deduce from Lemma~\ref{lem:GeodInQM} that the edges of a geodesic connecting $x$ to $y$ are all labelled by $V(\Lambda)$. Since the hyperplanes separating $x$ and $y$ coincide with the hyperplanes containing these edges as a consequence of Theorem~\ref{thm:BigQM}, it follows that the hyperplanes separating $x$ and $y$ are all labelled by $V(\Lambda)$. But $x,y$ and $a,b$ are separated by exactly the same hyperplanes, which implies that the edges of a geodesic connecting $a$ to $b$ also must be all labelled by $V(\Lambda)$, which implies that $a^{-1}b \in \langle \Lambda \rangle$. Thus, $a,b \in a \langle \Lambda \rangle$ with $a \langle \Lambda \rangle \in \mathcal{P}$ since $\langle \Lambda \rangle$ has polynomial growth. 

\medskip \noindent
Thus, we have proved that Proposition~\ref{prop:SyllabicQM} applies, which concludes the proof of our claim. 
\end{proof}

\addcontentsline{toc}{section}{References}

\bibliographystyle{alpha}
{\footnotesize\bibliography{PolynomiallyHyp}}

\begin{thebibliography}{BMW94}

\bibitem[BBP26]{Oussama}
S.~Bader, O.~Bensaid, and H.~Petyt.
\newblock Quasiisometric embeddings between right-angled {A}rtin groups:
  rigidity.
\newblock {\em arxiv:2605.12300}, 2026.

\bibitem[BC96]{MR1379364}
H.-J. Bandelt and V.~Chepoi.
\newblock Graphs of acyclic cubical complexes.
\newblock volume~17, pages 113--120. 1996.
\newblock Discrete metric spaces (Bielefeld, 1994).

\bibitem[BD92]{MR1153934}
H.-J. Bandelt and A.~Dress.
\newblock A canonical decomposition theory for metrics on a finite set.
\newblock {\em Adv. Math.}, 92(1):47--105, 1992.

\bibitem[BH99]{MR1744486}
M.~Bridson and A.~Haefliger.
\newblock {\em Metric spaces of non-positive curvature}, volume 319 of {\em
  Grundlehren der mathematischen Wissenschaften [Fundamental Principles of
  Mathematical Sciences]}.
\newblock Springer-Verlag, Berlin, 1999.

\bibitem[BMW94]{MR1297190}
H.-J. Bandelt, H.~Mulder, and E.~Wilkeit.
\newblock Quasi-median graphs and algebras.
\newblock {\em J. Graph Theory}, 18(7):681--703, 1994.

\bibitem[Bow14]{Bowditchcriterion}
B.~Bowditch.
\newblock Uniform hyperbolicity of the curve graphs.
\newblock {\em Pacific J. Math.}, 269:269--280, 2014.

\bibitem[Bro06]{BrownGroupoidTopology}
R.~Brown.
\newblock {\em Topology and groupoids}.
\newblock BookSurge, 2006.

\bibitem[Car15]{MR3358258}
M.~Carr.
\newblock {\em Two-generator subgroups of right-angled {A}rtin groups are
  quasi-isometrically embedded}.
\newblock ProQuest LLC, Ann Arbor, MI, 2015.
\newblock Thesis (Ph.D.)--Brandeis University.

\bibitem[CH09]{MR2541383}
P.-E. Caprace and F.~Haglund.
\newblock On geometric flats in the {CAT}(0) realization of {C}oxeter groups
  and {T}its buildings.
\newblock {\em Canad. J. Math.}, 61(4):740--761, 2009.

\bibitem[Che00]{MR1748966}
V.~Chepoi.
\newblock Graphs of some {${\rm CAT}(0)$} complexes.
\newblock {\em Adv. in Appl. Math.}, 24(2):125--179, 2000.

\bibitem[Gen17]{QM}
A.~Genevois.
\newblock Cubical-like geometry of quasi-median graphs and applications to
  geometric group theory.
\newblock {\em PhD thesis, Universit\'e Aix-Marseille, arxiv:1712.01618}, 2017.

\bibitem[Gen19]{VanKampenGP}
A.~Genevois.
\newblock On the geometry of van {K}ampen diagrams of graph products of groups.
\newblock {\em arXiv:1901.04538}, 2019.

\bibitem[Gen21a]{SpecialRH}
A.~Genevois.
\newblock Algebraic characterisation of relatively hyperbolic special groups.
\newblock {\em Israel J. Math.}, 241(1):301--341, 2021.

\bibitem[Gen21b]{coningoff}
A.~Genevois.
\newblock Coning-off {$\rm CAT(0)$} cube complexes.
\newblock {\em Ann. Inst. Fourier (Grenoble)}, 71(4):1535--1599, 2021.

\bibitem[Gen21c]{MR4227231}
A.~Genevois.
\newblock Negative curvature in graph braid groups.
\newblock {\em Internat. J. Algebra Comput.}, 31(1):81--116, 2021.

\bibitem[Gen22]{Mediangle}
A.~Genevois.
\newblock Rotation groups, mediangle graphs, and periagroups: a unified point
  of view on {C}oxeter groups and graph products of groups.
\newblock {\em arXiv:2212.06421}, 2022.

\bibitem[Gen23]{MR4586831}
A.~Genevois.
\newblock Special cube complexes revisited: a quasi-median generalization.
\newblock {\em Canad. J. Math.}, 75(3):743--777, 2023.

\bibitem[Gen24]{MR4808711}
A.~Genevois.
\newblock Automorphisms of graph products of groups and acylindrical
  hyperbolicity.
\newblock {\em Mem. Amer. Math. Soc.}, 301(1509):vi+127, 2024.

\bibitem[Gen25]{MR4922688}
A.~Genevois.
\newblock Cyclic hyperbolicity in {${\rm CAT}(0)$} cube complexes.
\newblock {\em Fund. Math.}, 269(3):201--248, 2025.

\bibitem[Ger98]{MR1663779}
V.~Gerasimov.
\newblock Fixed-point-free actions on cubings [translation of {\it {a}lgebra,
  geometry, analysis and mathematical physics ({r}ussian) ({n}ovosibirsk,
  1996)}, 91--109, 190, {I}zdat. {R}oss. {A}kad. {N}auk {S}ibirsk. {O}tdel.
  {I}nst. {M}at., {N}ovosibirsk, 1997; {MR}1624115 (99c:20049)].
\newblock {\em Siberian Adv. Math.}, 8(3):36--58, 1998.

\bibitem[GM08]{MR2448064}
D.~Groves and J.~Manning.
\newblock Dehn filling in relatively hyperbolic groups.
\newblock {\em Israel J. Math.}, 168:317--429, 2008.

\bibitem[Gre90]{GreenGP}
E.~Green.
\newblock Graph products of groups.
\newblock {\em PhD Thesis}, 1990.

\bibitem[Gro87]{MR919829}
M.~Gromov.
\newblock Hyperbolic groups.
\newblock In {\em Essays in group theory}, volume~8 of {\em Math. Sci. Res.
  Inst. Publ.}, pages 75--263. Springer, New York, 1987.

\bibitem[Hag14]{MR3217625}
M.~Hagen.
\newblock Weak hyperbolicity of cube complexes and quasi-arboreal groups.
\newblock {\em J. Topol.}, 7(2):385--418, 2014.

\bibitem[HS17]{MR3741855}
D.~Hume and A.~Sisto.
\newblock Groups with no coarse embeddings into hyperbolic groups.
\newblock {\em New York J. Math.}, 23:1657--1670, 2017.

\bibitem[HW99]{HsuWise}
T.~Hsu and D.~Wise.
\newblock On linear and residual properties of graph products.
\newblock {\em Michigan Math. J.}, 46:251--259, 1999.

\bibitem[HW08]{MR2377497}
F.~Haglund and D.~Wise.
\newblock Special cube complexes.
\newblock {\em Geom. Funct. Anal.}, 17(5):1551--1620, 2008.

\bibitem[Kam09]{MR2475886}
M.~Kambites.
\newblock On commuting elements and embeddings of graph groups and monoids.
\newblock {\em Proc. Edinb. Math. Soc. (2)}, 52(1):155--170, 2009.

\bibitem[KM02]{MR1920184}
S.~Klav\v{z}ar and H.~Mulder.
\newblock Partial cubes and crossing graphs.
\newblock {\em SIAM J. Discrete Math.}, 15(2):235--251, 2002.

\bibitem[MO15]{MR3368093}
A.~Minasyan and D.~Osin.
\newblock Acylindrical hyperbolicity of groups acting on trees.
\newblock {\em Math. Ann.}, 362(3-4):1055--1105, 2015.

\bibitem[Rat05]{MR2174099}
D.~Rattaggi.
\newblock Anti-tori in square complex groups.
\newblock {\em Geom. Dedicata}, 114:189--207, 2005.

\bibitem[Rol98]{Roller}
M.~Roller.
\newblock Pocsets, median algebras and group actions; an extended study of
  dunwoody's construction and sageev's theorem.
\newblock {\em dissertation}, 1998.

\bibitem[Sag95]{MR1347406}
M.~Sageev.
\newblock Ends of group pairs and non-positively curved cube complexes.
\newblock {\em Proc. London Math. Soc. (3)}, 71(3):585--617, 1995.

\bibitem[ST13]{MR3079268}
M.~Stein and J.~Taback.
\newblock Metric properties of {D}iestel-{L}eader groups.
\newblock {\em Michigan Math. J.}, 62(2):365--386, 2013.

\end{thebibliography}

\Address

%

\end{document}